\documentclass[a4paper,11pt]{article}
\usepackage{amssymb,amsmath,amsthm}
\usepackage{verbatim}  
\usepackage{graphicx}
\usepackage{enumerate}
\usepackage[all]{xy}  
\usepackage{tikz}
\usepackage[pdftex]{hyperref}
\hypersetup{colorlinks, citecolor=black, filecolor=black,linkcolor=black, urlcolor=black}
\usepackage{color}

\newtheorem{theorem}{Theorem} [section]
\newtheorem{lemma}[theorem]{Lemma}
\newtheorem{proposition}[theorem]{Proposition}
\newtheorem{corollary}[theorem]{Corollary}

\newtheorem{thmx}{Theorem}

\theoremstyle{definition}
\newtheorem{definition}[theorem]{Definition}
\newtheorem{notation}[theorem]{Notation}
\newtheorem{example}[theorem]{Example}
\newtheorem{remark}[theorem]{Remark}

\newcommand{\R}{\mathbb{R}}
\newcommand{\N}{\mathbb{N}}
\newcommand{\Z}{\mathbb{Z}}
\newcommand{\Co}{\mathbb{C}}

\newcommand{\emptyword}{\lambda}  
\newcommand{\groupid}{\epsilon}  

\newcommand{\D}{{B}}            
\newcommand{\RC}{\mathsf{RC}}   
\newcommand{\Log}{\log}

\newcommand{\tim}{}   

\newcommand{\fix}{\mathsf{Fix}}
\newcommand{\subfix}[2]{\fix_{#1}(#2)}
\newcommand{\Necklaces}{\mathsf{Necklaces}}      
\newcommand{\supp}{\mathsf{Supp}}

\newcommand{\sgf}{\theta}        
\newcommand{\sgs}{\sigma}         
\newcommand{\cgs}{\tilde{\sigma}} 
\newcommand{\egr}{\mathsf{gr}}  
\newcommand{\egrg}{\rho}         
\newcommand{\egrc}{\tilde{\rho}} 

\newcommand{\neck}{\mathsf{N}} 

\newcommand{\Ne}{\mathsf{Lk}}        
\newcommand{\cli}{\mathsf{Clq}}      

\newcommand{\Pa}{\mathcal{P}}  
\newcommand{\Ri}{\mathcal{R}}   

\newcommand{\wordl}[1]{|#1|}  
\newcommand{\eltl}[1]{||#1||}  
\newcommand{\diag}{\mathsf{diag}}  

\newcommand{\shuff}{\stackrel{les}{\rightarrow}}
\newcommand{\lrles}{\stackrel{lrles}{\rightarrow}}

\newcommand{\opsx}{\stackrel{lrlescrcp}{\longrightarrow}}

\newcommand{\xtoy}{\zeta}  

\newcommand\geol{\mathsf{Geo}}   
   
\newcommand\slex{<_{sl}}
\newcommand\slexeq{\leq_{sl}}

\newcommand\cycperm{\mathsf{CycPerm}}

\newcommand\sphl{\mathsf{SL}}    
\newcommand\wsphl{{shortlex language}}    
\newcommand\sphs{{\sigma}}    
\newcommand\wsphs{{spherical growth series}}    

\newcommand\geocl{\mathsf{ConjGeo}}   
   
\newcommand\geocpl{\mathsf{CycGeo}}   

\newcommand\sphcl{\mathsf{ConjSL}}  
\newcommand\wsphcl{{shortlex conjugacy language}}    
\newcommand\sphcs{{\widetilde {\sigma}}}    
\newcommand\wsphcs{{spherical conjugacy growth series}}    

\newcommand\slnf{{shortlex normal form}}
\newcommand\slcnf{{shortlex conjugacy normal form}}

\begin{document}

\title{Formal conjugacy growth in graph products I}
\author{L. Ciobanu, S. Hermiller, V. Mercier}
\maketitle

\begin{abstract}

In this paper we give a recursive formula for the conjugacy growth series of a graph product in terms of the conjugacy growth and standard growth series of subgraph products. We also show that the conjugacy and standard growth rates in a graph product are equal provided that this property holds for each vertex group. All results are obtained for the standard generating set consisting of the union of generating sets of the vertex groups.

\bigskip


\bigskip

\end{abstract}


\section{Introduction}\label{sec:intro}

In this paper we obtain several results on conjugacy growth and languages in graph products with respect to their standard generating set: foremost, we find a formula for the conjugacy growth series for a graph product of groups as a function of the standard and conjugacy growth series of subgraphproducts, and in parallel we establish the equality of the standard and conjugacy growth rates if the same holds in each vertex group. En route to proving these results we also study the shortlex conjugacy language for graph products.

The graph product construction generalizes both direct
and free products. Given a finite simplicial graph 
with vertex set $V$ and for each vertex $v\in V$
an associated group $G_v$, the associated \textit{graph product} $G_V$ is the
group generated by the vertex groups with the added relations that
elements of groups attached to adjacent vertices commute.
Right-angled Artin groups (RAAGs) and Coxeter groups (RACGs) arise in this
way, as the graph products of infinite cyclic groups and cyclic groups of order 2
respectively, and have been widely studied.
Graph products were introduced by Green in her PhD thesis \cite{G90} and their (standard) growth series, based on the growth series of the vertex groups, were subsequently computed by Chiswell~\cite{C94}.

The first conjugacy growth series computations appeared in the work of Rivin (\cite{R04, R10}) on free groups, and it is striking that, even for free groups with standard generating sets, the series are transcendental, and their formulas rather complicated. More generally and systematically, conjugacy growth series and languages featured in \cite{AC17,CH14, CHHR14,BdlH16, Evetts19, VM17}, where virtually abelian groups, acylindrically hyperbolic groups, free and wreath products, and more, were explored. 

All groups in this paper are finitely generated, and all generating sets finite and inverse-closed. 
The \textit{spherical}, or \textit{standard, growth function} of a group $G$ with respect to a generating set $X$ records the size of the sphere of radius $n$ in the Cayley graph of $G$ with respect to $X$ for each $n\geq 0$, and the \textit{spherical conjugacy growth function} counts the number of conjugacy classes intersecting the sphere of radius $n$ but not the ball of radius $n-1$. Taking the 
growth rate of the values given by the above functions produces the \textit{spherical growth rate} and \textit{spherical conjugacy growth rate} of $G$ with respect to $X$. 
Furthermore, the spherical standard growth and conjugacy growth series are those generating functions whose coefficients are the spherical growth function and spherical conjugacy growth function values, respectively.
The exact meaning of the terminology used below, and all necessary notation, is given in Section~\ref{subsec:notation}.

The first main result of the paper gives a recursive formula for the spherical conjugacy growth series of a graph product, based on the spherical growth and conjugacy growth series of the vertex groups.  

\begin{thmx}\label{thmx:the_formula_for_cgs}
Let $G_V$ be a graph product group over a graph
with vertex set $V$ and let $v\in V$ be a vertex. 
For each $v' \in V$ let $X_{v'}$ be an inverse-closed generating
set for the vertex group $G_{v'}$. 
For each subset $S \subseteq V$ let $X_S = \cup_{v' \in S}X_{v'}$ be the generating set for the 
subgraph product $G_S$ on the subgraph 
induced by $S$. 
Let $\cgs_S$ be the spherical
conjugacy growth series and let $\sgs_S$ be the spherical
growth series of $G_S$ with respect to $X_S$.
 
 Then the 
conjugacy growth series of $G_V$ is given by
\[  \cgs_V=\, \cgs_{V\setminus\{v\}}+\cgs_{\Ne(v)}(\cgs_{\{v\}}-1) +
   \sum_{S\subseteq \Ne(v)} \cgs^{\mathcal{M}}_S~ 
  \neck\left(\left(\frac{\sgs_{\Ne(S)\setminus\{v\}}}
    {\sgs_{\Ne(v)\cap\Ne(S)}}-1\right)(\sgs_{\{v\}}-1)\right),\]
where 
$\Ne(v)$ is the set of vertices adjacent to $v$,
$\sphcs^{\mathcal{M}}_{S} = \sum_{S'\subseteq S} (-1)^{|S|-|S'|} \sphcs_{S'}$,
 and for any complex power series $f(z)$,
\[ \neck(f)(z):=\sum_{k=1}^\infty\sum_{l=1}^\infty \frac{\phi(k)}{kl}\big(f(z^k)\big)^l\]
in which $\phi$ is the Euler totient function.

Moreover, if $\{v\} \cup \Ne(v) = V$, then
$\cgs_V =  \cgs_{\Ne(v)}\cgs_{\{v\}}$. 
\end{thmx}

The proof of Theorem~\ref{thmx:the_formula_for_cgs} employs
the use of M\"obius inversion formulas applied
to languages of conjugacy representatives 
that arise from the amalgamated free product decomposition of a graph product.  
The second main result of the paper follows from many
of the same techniques and
shows that equality of the spherical growth and
conjugacy growth rates is preserved by the graph product 
construction.

\begin{thmx}\label{thmx:equal_rates}
Let $G_V$ be a graph product group 
over a vertex set $V$ and assume that for each vertex $v\in V$ 
the spherical growth rate and spherical
conjugacy growth rate of $G_v$, over a generating set $X_v$, are equal.  Let $X_V=\cup_{v\in V} X_v$.
Then the spherical 
growth rate and 
spherical conjugacy growth rate of $G_V$ with respect to $X_V$ are equal. Hence also the
radii of convergence of the spherical and spherical conjugacy growth
series of $G_V$ over $X_V$ are equal.
\end{thmx}

It is interesting that many infinite discrete groups display the same behaviour as that in Theorem~\ref{thmx:equal_rates}, that is, the standard and conjugacy growth rates are equal. This is the case for hyperbolic~\cite{AC17} and relatively hyperbolic~\cite{gy18} groups, the wreath products (including lamplighter groups) in~\cite{VM17}, and soluble 
Baumslag-Solitar groups $BS(1,k)$~\cite{ceh20}. It is an intriguing question whether the equality of growth rates holds for larger classes of groups (such as acylindrically hyperbolic), or if there exists a common thread in the proofs of this equality for the different classes of groups mentioned above.

The proofs of the two main theorems revolve around methods that come from analytic combinatorics, such as the `necklace' series associated to a language.
Since these tools are not standard in group theory, we begin in
Section~\ref{sec:preliminaries} with a discussion of these tools.
In Section~\ref{sec:grprod} we provide background information
as well as new results on languages associated to
 graph products of groups,
including conjugacy and cyclic geodesics, and
shortlex normal forms for conjugacy classes,
that are used in the rest of the paper.
 
In Section~\ref{subsec:conjgeoreps} we establish in 
Proposition~\ref{prop:UC_pearls}
a set of conjugacy geodesics (minimal length representatives,
over the generators,
for conjugacy classes) 
for a graph product group
that contains at least one
representative for each conjugacy class, and we determine
when two elements of this set represent conjugate elements.
The remainder of Section~\ref{sec:conj_growth_series} 
contains the proofs of Theorems~\ref{thmx:the_formula_for_cgs} 
and~\ref{thmx:equal_rates}, as well as Example \ref{example_RACG}, where the conjugacy growth series of a right-angled Coxeter group is computed using the formulas in the paper.

Further types of formulas for the spherical conjugacy growth series of graph products and an analysis of their algebraic complexity will be the subject of a subsequent paper.


\section{Preliminaries and necklace languages}\label{sec:preliminaries}


\subsection{Notation and terminology}\label{subsec:notation}

We use standard notation from formal language theory:
 where $X$ is a finite set,
we denote by $X^*$ the set of all words over $X$, and call a subset of
$X^*$ a language.
We write $\emptyword$ for the empty word, 
and denote by $X^+$ the set of all non-empty words over $X$ (so $X^*=X^+ \cup \{ \emptyword \}$).
For each word $w \in X^*$, let $l(w)=l_X(w)=\wordl{w}$ denote its length over $X$.

For a group $G$ with inverse-closed generating set $X$, let $\pi:X^* \rightarrow G$ be the natural projection onto $G$, and let
$=$ denote equality between words and $=_G$ equality between group elements
(so $w=_G v$ means $\pi(w)=\pi(v)$).
For $g \in G$, the {\em length}
of $g$, denoted $\eltl{g}\;(=\eltl{g}_X)$, is
the length of a shortest representative word for $g$ over $X$.
A {\em geodesic} is a word $w \in X^*$ with $l(w)=\eltl{\pi(w)}$;
we denote the set of all geodesics for $G$ with
respect to $X$ by $\geol(G,X)$.

Let $\sim$, or $\sim_G$, denote the equivalence relation on $G$ given
by conjugacy, and $G/\sim$ its set of equivalence classes.
Let $[g]_{\sim}$ denote the conjugacy class of $g \in G$ and $\eltl{g}_{\sim}$ denote
its {\em length up to conjugacy}, that is, \[\eltl{g}_{\sim}:=\min\{\eltl{h} \mid h \in [g]_{\sim}\}.\]
We say that $g$ has {\em minimal length up to conjugacy}
if $\eltl{g}=\eltl{g}_{\sim}$.
A {\em conjugacy geodesic} is a word $w \in X^*$
with $l(w) = \eltl{\pi(w)}_{\sim}$; we denote
the set of all conjugacy geodesics by $\geocl(G,X)$.
 
Fix a total ordering of $X$, and 
let $\slexeq$ be 
the induced shortlex ordering of $X^*$ (for which
$u \slex w$ if either $l(u)<l(w)$, or $l(u)=l(w)$ but
$u$ precedes $w$ lexicographically).
For each $g \in G$, the {\em \slnf} of $g$ is the unique word
$y_g \in X^*$ with $\pi(y_g)=g$ such that
$y_g \slexeq w$  for all $w \in X^*$
with $\pi(w)=g$.
For each conjugacy class $c \in G/\sim$, the 
{\em \slcnf} of $c$ 
is the shortlex least word $z_c$ over $X$ representing an 
element of $c$; that is, 
$\pi(z_c)\in c$, and $z_c \slexeq w$ for all 
$w \in X^*$ with $\pi(w)\in c$.
The {\em \wsphl}
 and {\em \wsphcl} for $G$ over $X$ are defined as 
\begin{eqnarray*}
\sphl=\sphl(G,X)&:=&\{y_g \mid g \in G\},\\
\sphcl=\sphcl(G,X)&:=&\{z_c \mid c \in G/\sim\}.\end{eqnarray*}

Any language $L$ over $X$ gives rise to a
{\em strict growth function} $\sgf_L:\N \cup \{0\} \to \N \cup \{0\}$, 
defined by  $\sgf_L(n) := |\{w \in L \mid l(w) = n\}|$;
an associated generating function, called
the {\em strict growth series}, given
by $F_L(z) := \sum_{n=0}^\infty \sgf_L(n)z^n$; and  
an {\em exponential} 
{\em growth rate}
$\egr_L = \lim_{n \to \infty} (\sgf_L(n))^{1/n}$.

For the two languages above, the coefficient
$\sgf_{\sphl}(n)$ is the number of elements of $G$  of length $n$,
  and $\sgf_{\sphcl}(n)$ is the number of conjugacy classes of $G$ whose shortest
elements have length $n$.
As in ~\cite{CH14}, we
refer to the strict growth series of $\sphl$ below as the \emph{standard} or {\em \wsphs} 
\[ \sphs(z)=\sphs_{(G,X)}(z) := F_{\sphl(G,X)}(z)
  = \sum_{n=0}^\infty \sgf_{\sphl(G,X)}(n)z^n \]
and the strict growth series 
\[ \sphcs(z)=\sphcs_{(G,X)}(z) := F_{\sphcl(G,X)}(z)
  = \sum_{n=0}^\infty \sgf_{\sphcl(G,X)}(n)z^n, \]
of $\sphcl$ as the {\em \wsphcs}.

\begin{remark}
Note that the growth series in the paper will be often denoted as $\sphs$ and $\sphcs$ instead of $\sphs(z)$ or $\sphcs(z)$ due to the length of some of the formulas.
\end{remark}

For the group $G$ and generating set $X$ the 
exponential 
growth rates of these two series give
the \emph{standard} or {\em spherical growth rate} of a group $G$ over $X$, namely
\begin{equation}
\egrg = \egrg(G,X) := \egr_{\sphl(G,X)} 
  = \lim_{n \to \infty} (\sgf_{\sphl(G,X)}(n))^{1/n},
 \end{equation} 
and the {\em spherical conjugacy growth rate}, given by
\begin{equation}
\egrc = \egrc(G,X) := \egr_{\sphcl(G,X)}
  = \limsup_{n \to \infty} (\sgf_{\sphcl(G,X)}(n))^{1/n}.
  \end{equation}



\subsection{Complex power series}\label{subsec:cxseries}


In this section we recall some basic facts about power series in complex analysis (see for example~\cite[Chapter III Sections 1-2]{C78}).

 We denote the open disk of radius $r>0$ centered at $c \in \Co$ by
$
  \D(c,r):=\{z\in\Co\,:\,|z-c|<r\}.
 $
A {\em complex power series} is a function $f:\D(0,r)\rightarrow \Co$ of the form
$
  f(z)=\sum_{n=0}^{\infty}a_{n}z^n,
 $
where $a_{n}\in\Co$ for all $n$. We express the fact that $a_n$ is the {\em coefficient} of $z^n$ in $f$ by writing
\[
 [z^n]f(z):=a_n.
\]
The radius of convergence $\RC(f)$ of $f$ can be defined as 
$$
 \RC(f)=\sup\{r\in\R\,:\,f(z)\textup{ converges for all }z\in\D(0,r)\},
$$ 
or equivalently as 
\begin{equation}\label{growth_rate}
 \RC(f)=\frac{1}{\limsup_{n\to\infty}\root n\of{|a_{n}|}}.
\end{equation}
If $\RC(f)>0$, then $f$ is defined and converges
absolutely at every point in the open disk $\D(0,\RC(f))$.

\begin{proposition}\label{prop:positive_preim_1}
 Let $f\neq 0$ be a complex power series such that $\RC(f)>0$, $[z^n]f(z) \geq 0$ for all $n \in \N$, and $[z^0]f(z)=0$. Then there exists a unique positive number $t>0$ such that $f(t)=1$; moreover,
 \[
  t=\inf\{|z|\,:\,z \in \Co,~|f(z)|=1\}=\sup\{r>0\,:\,
|f(z)|\leq 1 \textup{ for all } z\in\D(0,r)\},
 \]
 and the infimum and supremum are attained.
\end{proposition}

\begin{proof}
Write $f=\sum_{n=1}^\infty a_n z^n$, where $a_n \ge 0$
for all $n$ and $a_m \neq 0$ for at least one index $m$.
For any complex number $z$ we have
\[
|f(z)| = |\sum_{n=1}^\infty a_n z^n| \le \sum_{n=1}^\infty a_n |z|^n = f(|z|);
\]
hence if the series $f(z)$ diverges,
then so does the series $f(|z|)$.
Moreover $f(|z|) \ge a_m |z|^m$ is unbounded
as $|z|$ increases.
Thus on the 
real interval $[0,\RC(f))$ 
the function $f$ is continuous, strictly increasing, and unbounded, and so there exists a unique $t \in [0,\RC(f))$ such that $f(t)=1$. Now for any complex number $z$ satisfying $|z|\leq t$ the following holds:
\[
 |f(z)| \leq \sum_{n=1}^\infty a_n |z|^n \leq \sum_{n=1}^\infty a_n t^n=f(t)=1.	
\]
\end{proof}


\subsection{Necklace set associated to a language}\label{subsec:necklaces}


Let $X$ be a finite alphabet and $L$ be a language over $X$. 
Let $\N$ denote the positive integers, and 
$\N_0$ denote the nonnegative integers.
For $n \in \N$,
let $L^{\tim n}$ denote the Cartesian product of $n$ copies of $L$.
For $(l_1,\ldots,l_n)\in L^{\tim n}$, the elements $l_j$ with $j\in\{1,\ldots,n\}$ are called the {\em components}, and the
{\em length} of this $n$-tuple is defined to be $\wordl{(l_1,\ldots,l_n)}:=\sum_{j=1}^n \wordl{l_j}$.

Let $C_n:=\Z/n\Z$. 
The group $C_n$ acts on $L^{\tim n}$ by cyclically permuting the entries of tuples in $L^{\tim n}$, that is,
$g \cdot (u_1,...,u_n) := (u_{1+g},...,u_{n+g})$ for
all $g \in C_n$ and $u_1,...,u_n \in L$, where the
index $i+g$ of $u_{i+g}$ is taken modulo $n$.
Let $L^{\tim n}/C_n$ denote the quotient by this action, and
define the set of \emph{necklaces over $L$} as
 \[
  \Necklaces (L):=
\bigcup_{n=1}^\infty \big(L^{\tim n} / C_n \big).
 \]
Since the length of an element in $L^{\tim n}$ is preserved by cyclic permutation of its components, we extend the definition of length on $L^{\tim n}$ to $\Necklaces(L)$.

In analogy with the growth of languages over an alphabet,
any set $S$ together with a length function 
$\wordl{\cdot} : S \to \N_0$, satisfying the
property that for each nonnegative integer the
number of elements of that length is finite, 
has a strict growth function
$\sgf_S:\N_0 \to \N_0$, 
defined by  $\sgf_S(n) := |\{s \in S \mid \wordl{s} = n\}|$, and
a strict growth series given
by $F_S(z) := \sum_{n=0}^\infty \sgf_S(n)z^n$.

Next we collect some identities among several strict growth series.  
Given $u \in L$, let $\diag(u)$ denote the diagonal
element $\diag_n(u):=(u,u,...,u)$ in $L^{\tim n}$.   
Similarly, for 
$v=(v_1,...,v_d) \in L^{\tim d}$ and $m \in \N$, let $\diag_m(v)$ denote the
element $\diag_m(v):=(v_1,...,v_d,...,v_1,...,v_d)$ of $L^{\tim md}$.
Note that whenever $n\neq n'$, the sets
$L^{\tim n}/C_n$ and $L^{\tim n'}/C_{n'}$ are disjoint.

\begin{lemma} \label{lem:seriesidentities}
Let $L$ be a language and let $n\in \N$.  Then the following hold.
\begin{enumerate}
\item\label{lem:necksum} $F_{\Necklaces(L)}(z) = 
    \sum_{n=1}^{\infty} F_{L^{\tim n}/C_n}(z)$.
\item\label{lem:sipower}  $F_{L^{\tim n}}(z)=\big(F_L(z)\big)^n$.
\item  $F_{\{\diag_m(u)\,:\,u \in L^{\tim d}\}}(z)=F_{L^{\tim d}}(z^m)$.  
\item\label{lem:sicoeff} $[z^{m}](F_L(z))^d = [z^{mn}](F_L(z^n))^d$.
\end{enumerate}
\end{lemma}

The following gives a computation of the strict growth
series $F_{\Necklaces(L)}(z)$ from $F_L(z)$.

\begin{proposition}\label{prop:neckseriesidentity}
 The growth series of the set of necklaces over a language $L$ is 
 \begin{align*}
 F_{\Necklaces(L)}(z)=& \sum_{k=1}^\infty\sum_{l=1}^\infty \frac{\phi(k)}{kl}\big(F_L(z^k)\big)^l,
\end{align*}
where $\phi$ is the Euler totient function.
\end{proposition}

\begin{proof}
For every $n,m\in\N$, the set 
$S^n(m):=\{w\in L^{\tim n}\,:\,\wordl{w}=m\}$
is invariant under the cyclic permutation action of $C_n$ on $L^{\tim n}$.  
Then the coefficient $[z^m]F_{L^{\tim n}/C_n}(z)$ is the number of orbits in $S^n(m)$ under the action of $C_n$. 
For each $g \in C_n$, 
let $\subfix{S^n(m)}{g}$ denote the set of elements of $S^n(m)$
that are fixed by the action of $g$.
Using Burnside's Lemma we find
\[
 [z^m]F_{L^{\tim n}/C_n}(z)=\frac{1}{n}\sum_{g\in C_n}|\subfix{S^n(m)}{g}|=\frac{1}{n}\sum_{d|n}\sum_{\stackrel{1\leq g\leq n}{(g,n)=d}}|\subfix{S^n(m)}{g}|.
\]
In fact whenever $d|n$, $1\leq g\leq n$, $(g,n)=d$, and $w \in L^{\tim n}$,
then 
$w\in \subfix{S^n(m)}{g}$ if and only if $w=\diag_{\frac{n}{d}}(v)$ for some $v\in L^{\tim d}$ with $|v|=\frac{md}{n}$. 
In the case that $(g,n)=d$, then
\[ |\subfix{S^n(m)}{g}|=[z^{\frac{md}{n}}]\big(F_{L^{\tim d}}(z)\big)=
  [z^{\frac{md}{n}}]\big(F_L(z)\big)^d=[z^m]\big(F_L(z^{\frac{n}{d}})\big)^d
\]
where the second and third equalities apply 
parts~\ref{lem:sipower} and~\ref{lem:sicoeff} of Lemma~\ref{lem:seriesidentities}
respectively.  Therefore we find
\[
 F_{L^{\tim n}/C_n}(z)=\frac{1}{n}\sum_{d|n}|\{1\leq g\leq n,\,(g,n)=d\}|\big(F_L(z^{\frac{n}{d}})\big)^d=\frac{1}{n}\sum_{d|n}\phi(\frac{n}{d})\big(F_L(z^{\frac{n}{d}})\big)^d.
\]
Finally, using Lemma~\ref{lem:seriesidentities} part~\ref{lem:necksum},
\begin{align*}
 F_{\Necklaces (L)}(z)=& \sum_{n=1}^\infty \frac{1}{n}\sum_{d|n}\phi(\frac{n}{d})\big(F_L(z^{\frac{n}{d}})\big)^d=\sum_{k=1}^\infty\sum_{l=1}^\infty \frac{\phi(k)}{kl}\big(F_L(z^k)\big)^l.
\end{align*}
\end{proof}

Note that if the language $L$ contains the empty word 
then the set $\Necklaces(L)$ contains infinitely
many elements of length $0$ and so the strict growth
series $F_{\Necklaces(L)}(z)$ is nowhere defined. 
Thus for the remainder of the paper, every language $L$ for which we 
consider the series
$F_{\Necklaces (L)}(z)$ is assumed not to contain the empty word,
so that $F_L(0) = 0$.

\begin{remark}\label{rmk:bound_neck_coef_cycl_perm}
For every $n\in\N$ and $m\in\N_0$, the strict
growth functions for $L^{\tim n}$ and $L^{\tim n}/C_n$ satisfy
$\frac{1}{n}\sgf_{L^{\tim n}}(m) \le \sgf_{L^{\tim n}/C_n}(m)
\le \sgf_{L^{\tim n}}(m)$.  Then Lemma~\ref{lem:seriesidentities}
parts~\ref{lem:necksum},\ref{lem:sipower} yield
 \[
  [z^m]\sum_{n=1}^\infty \frac{\big(F_L(z))^n}{n}\leq[z^m]F_{\Necklaces(L)}(z)\leq[z^m]\underbrace{\sum_{n=1}^\infty \big(F_L(z))^n}_{\frac{F_L(z)}{1-F_L(z)}} .
 \]
\end{remark}

\begin{corollary}\label{cor_RC_necklace}
Let $L$ be a nonempty language that does not contain the empty word.
 The radius of convergence of $F_{\Necklaces(L)}(z)$ is given by
 \[
  \RC(F_{\Necklaces(L)}(z))=\inf\{|z|\,:\,z\in\Co,\,|F_L(z)|=1\},
 \]
which is the positive real
number $t$ such that $F_L(t)=1$.
\end{corollary}

\begin{proof}
Remark~\ref{rmk:bound_neck_coef_cycl_perm} implies that
\[
 \RC\left(\sum_{n=1}^\infty \frac{\big({F_L(z))}^n}{n}\right)\geq \RC(F_{\Necklaces(L)}(z))\geq \RC\left(\sum_{n=1}^\infty \big({F_L(z))}^n \right).
\]

The convergence radius of the geometric series $\sum_{n>0}z^n$ is 1, and so 
the series $\sum_{n=1}^\infty \big({F_L(z))}^n$ converges for all $z$ satisfying $|F_L(z)|<1$ and diverges for all $z$ such that $|F_L(z)|>1$. 
Since the language $L$ is a subset of $X^*$ for a finite set $X$,
we have $\sgf_L(m) \le |X|^m$ for all $m$, and so
the radius of convergence of the strict growth series
$F_L$ is at least $\frac{1}{|X|}$.
Hence, using Proposition~\ref{prop:positive_preim_1},
\begin{equation}\label{eq_RC_necklace} 
\begin{split}
  \RC\left(\sum_{n=1}^\infty \big({F_L(z))}^n\right) & = \sup\{r>0\,:\,|F_L(z)|\leq 1 \textup{ for all } z\in\D(0,r)\} \\
 & = \min\{|z|\,:\,z\in\Co,\,|F_L(z)|=1\}.
\end{split}
\end{equation}
Therefore it suffices to prove that 
 \[
 \RC\left(\sum_{n=1}^\infty \frac{\big({F_L(z))}^n}{n}\right)=\RC\left(\sum_{n=1}^\infty \big({F_L(z))}^n \right).
 \]
Note that because $F_L(0)=0$,
 \[
  [z^m]\sum_{n=1}^\infty {F_L(z)}^n=[z^m]\sum_{n=1}^m {F_L(z)}^n=[z^m]\frac{1-{F_L(z)}^{m+1}}{1-F_L(z)},
 \]
and
\[
 [z^m]\sum_{n=1}^\infty \frac{{F_L(z)}^n}{n}=[z^m]\sum_{n=1}^m \frac{{F_L(z)}^n}{n}\geq[z^m] \sum_{n=1}^m \frac{{F_L(z)}^n}{m}=\frac{1}{m}[z^m]\frac{1-{F_L(z)}^{m+1}}{1-F_L(z)}.
\]
Hence
\begin{align*}
\frac{1}{ \RC\left(\sum_{n=1}^\infty \frac{{F_L(z)}^n}{n}\right)}=&\limsup_{m\to\infty}\root m \of{[z^m]\sum_{n=1}^\infty \frac{{F_L(z)}^n}{n}}\geq\limsup_{m\to\infty}\root m \of{\frac{1}{m}[z^m]\frac{1-{F_L(z)}^{m+1}}{1-F_L(z)}}\\
=&\limsup_{m\to\infty}\root m \of{[z^m]\frac{1-{F_L(z)}^{m+1}}{1-F_L(z)}}=\limsup_{m\to\infty}\root m \of{[z^m]\sum_{n=1}^\infty {F_L(z)}^n}\\
=&\frac{1}{\RC\left(\sum_{n=1}^\infty {F_L(z)}^n\right)}.
\end{align*}
Therefore
\[
 \RC\left(\sum_{n=1}^\infty \frac{\big({F_L(z))}^n}{n}\right)\leq \RC\left(\sum_{n=1}^\infty \big({F_L(z))}^n \right).
\]
\end{proof}

\begin{example}
 Let $L=\{c_1,\ldots,c_p\}$ be a finite subset of $X$; that is, $\wordl{c_i}=1$ for all $i$. Then $F_L(z)=pz$ and the set $L$ can be viewed as a set of colors. In this case
 Proposition~\ref{prop:neckseriesidentity} says that
 \[
  F_{\Necklaces(L)}(z)=\sum_{k=1}^\infty\sum_{l=1}^\infty \frac{\phi(k)}{kl}p^lz^{kl}.
 \]
The coefficient of $z^m$ in this series is the number of necklaces that we can make with $m$ pearls, all with a color in $L$.
\end{example}

Proposition~\ref{prop:neckseriesidentity} leads us to the following definition.

\begin{definition}\label{def:neck}
For any complex power series $f$ with integer
coefficients satisfying $[z^0]f(z)=0$, let  
 \[
  \neck(f)(z):=\sum_{k=1}^\infty\sum_{l=1}^\infty \frac{\phi(k)}{kl}\big(f(z^k)\big)^l=\sum_{k=1}^\infty \frac{-\phi(k)}{k}\Log(1-f(z^k)).
 \]
\end{definition}

We note that if $f = F_L$ is the growth series of a nonempty language $L$ that does not contain the empty word,
then by Proposition~\ref{prop:neckseriesidentity},
$N(f) = N(F_L) = F_{\Necklaces(L)}$, and by
Corollary~\ref{cor_RC_necklace}, the radius of convergence
$\RC(\neck(f))$ is the unique positive number $t$ 
such that $f(t)=1$.

\begin{example}\label{ex:neckofpower}
 If $f(z)=z^K$ with $K>0$, then $\neck(f)(z)=\frac{z^K}{1-z^K}$.
 (See~\cite[Lemma~1(1)]{schneider}.)
\end{example}

\section{Graph products: Background and languages}\label{sec:grprod}


Let $\Gamma=(V,E)$ be a finite simple graph with vertex set $V$ and edge set $E$; that is, a non-oriented graph without loops or multiple edges.

For any nonempty subset $V' \subseteq V$,
the {\em link} or {\em centralizing set} $\Ne(V')$ of $V'$ denotes
the set of all vertices of $\Gamma$ that
are adjacent to all of the vertices in $V'$.
That is, 
for any vertex $v \in V$ the set
\[
 \Ne(v):=\{w\in V\,:\,\{v,w\}\in E\}
\]
 is the set of neighbours of $v$, and 
 for any nonempty subset $ V'\subseteq V$, we have
 \[
  \Ne(V'):=\bigcap_{v\in V'}\Ne(v).
 \]
We also set $\Ne(\emptyset):=V$.

 For each vertex $v$ of $\Gamma$, let
$G_v$ be a nontrivial group. The {\em graph product of the groups} $G_v$ {\em with respect to}
$\Gamma$ is the quotient of their free product by the normal
closure of the set of relators $[g_v,g_w]$ for all $g_v \in G_v$, $g_w \in G_w$
for which $\{v,w\}$ is an edge of $\Gamma$.

Given a graph product group $G$ over a graph
$\Gamma=(V,E)$ and 
any subset $V' \subseteq V$, 
the {\em subgraph product} associated to $V'$ is the subgroup
$G_{V'} := \langle G_v \mid v \in V' \rangle$ of $G$.
By~\cite[Proposition~3.31]{G90}, $G_{V'}$ is isomorphic to the graph product of
the $G_v$ ($v \in V'$) on the induced subgraph of $\Gamma$ with vertex set $V'$.
Note that $G_V=G$ and $G_\emptyset$ is the trivial group.

Suppose that each vertex group $G_v$ of the graph product
has an inverse-closed generating set $X_v$. 
For each $V' \subseteq V$, let 
\[
X_{V'} := \cup_{v \in V'} X_v;
\] 
then $X_{V'}$ is an inverse-closed generating set for $G_{V'}$.
A {\em syllable} of a word $w \in X_V^*$ is a
subword $u$ of $w$ satisfying the properties
that $u \in X_v^+$ for some
$v \in V$ and $u$ is not contained in a strictly
longer subword of $w$ that also lies in $X_v^*$.

For each $v \in V$ let
$Y_v:=G_v \setminus \{\groupid\}$ denote
the generating set for the vertex group $G_v$, and 
denote the associated generating set for $G_{V'}$ by
\[
Y_{V'} := \cup_{v \in V'} G_v \setminus \{\groupid\}.
\]
Define a function $\xtoy:X_V^* \rightarrow Y_V^*$ 
by setting $\xtoy(w)$ to be the word obtained from
$w \in X_V^*$ by replacing each syllable $u \in X_v^+$
of $w$ by the element of $Y_v$
represented by $u$.

\begin{definition}\label{def:support}
 For an element $g\in G_V$, the {\em support} of $g$ is the set
 \[
  \supp(g):=\bigcap_{V'\subseteq V \text{ and }g\in G_{V'}} V'.
 \]
 For a word $w \in X_V^*$, the {\em support} $\supp(w)$
 of $w$ is the set of
all vertices $v$ for which a letter of $X_v$ appears in $w$.
\end{definition}


\subsection{Geodesic languages and word operations}\label{sub:geolang}

Over the generating set $Y_V$,
one can obtain
a geodesic representative of an element $g \in G_V$
from any other geodesic representative by iteratively swapping the order 
of consecutive letters from commuting vertex groups
(see \cite[Theorem~3.9]{G90} or~\cite[Proposition~3.3]{CH14}). 
The support of $g$ can be realized as the
set of all vertices $v$ for which a nontrivial element 
in $G_v$ appears in a 
geodesic word representative of $g$ over $Y_V$.

In~\cite{CH14} Ciobanu and Hermiller give
characterizations of the geodesics and conjugacy
geodesics over the generating set $X_V$
in a graph product group $G_V$ using 
a collection of homomorphisms.
For each $v \in V$, define a monoid homomorphism
$\pi_v=\pi^X_v:X_V^* \rightarrow (X_v \cup \{\$\})^*$, 
where $\$$ denotes a letter not in $X_v$, by
defining 
\[
\pi_v(a) := \begin{cases}
a & \mbox{if } a \in X_v \\
\$& \mbox{if } a \in X_{V \setminus(\Ne(v) \cup \{v\})} \\
1 & \mbox{if } a \in X_{\Ne(v)}.
\end{cases}
\]
For the generating set $Y_V$ of $G_V$, we denote
the associated map $\pi_v:Y_V^* \rightarrow (Y_v \cup \{\$\})^*$
by $\pi^Y_v$.

Given languages $L,L'$ over a finite set $X$, let
$LL' := \{uv : u \in L,~v \in L'\}$ (the concatenation
of $L$ with $L'$),
$L^+ := \cup_{n=1}^\infty L^n$ (where $L^n := L^{n-1}L$
for all $n$), and 
$L^* :=  L^+ \cup \{\emptyword\}$.  Also define  
\[
\cycperm(L) := \{vu : uv \in L\}
\]
to be the set of \emph{cyclic permutations} of words in $L$.

\begin{lemma}\emph{(\cite[Propositions~3.3,~3.5]{CH14})}\label{lem:gpgeocgeo}
The set of geodesics in the graph product group $G_V$
with respect to the generating set $X_V$ is
\[
\geol(G_V,X_V)= \cap_{v \in V}~ 
  \pi_v^{-1}(\geol(G_v,X_v)(\$ \geol(G_v,X_v))^*),
\]
and the set of conjugacy geodesics is
\[
\geocl(G_V,X_V)=\cap_{v \in V}~
  \pi_v^{-1}(\geocl(G_v,X_v) \cup \cycperm((\$ \geol(G_v,X_v))^+)).
\]
\end{lemma}

For a group $G$ with inverse-closed generating set $X$,
we say that a word $w \in X^*$ is {\em cyclically geodesic} over $X$
if every cyclic permutation of $w$ lies in $\geol(G,X)$, and   
we denote 
\[
\geocpl(G,X) := \{ \text{cyclically geodesic words for }G 
  \text{ over }X\}.
\]

For the generating set $Y_V$, the fact that
$\geol(G_v,Y_v)=\geocpl(G_v,Y_v)=\geocl(G_v,Y_v)=Y_v \cup \{\emptyword\}$
for any vertex $v \in V$ together with Lemma~\ref{lem:gpgeocgeo}
show that $\geocl(G_V,Y_V)=\geocpl(G_V,Y_V)$.
The following is also an immediate consequence of 
Lemma~\ref{lem:gpgeocgeo}.

\begin{corollary}\label{cor:geosubset}
Let $G_V$ be a graph product group with
generating set $X_V$ and let $V'$ be any subset of $V$.  Then
\begin{eqnarray*}
\geol(G_V,X_V) \cap X_{V'}^* &=& \geol(G_{V'},X_{V'})\\
\geocpl(G_V,X_V) \cap X_{V'}^* &=& \geocpl(G_{V'},X_{V'}) \text{ and}\\
\geocl(G_V,X_V) \cap X_{V'}^* &=& \geocl(G_{V'},X_{V'}).
\end{eqnarray*}
\end{corollary}

We consider two sets of operations on words over $X_V$.
The following (first) set of operations on words over $X_V$ preserve
the group element being represented.
\begin{itemize}
\item {\em Local reduction}: $yuz \rightarrow ywz$ with $y,z \in X_V^*$,
$u,w \in X_v^*$ for some $v \in V$, $u=_{G_v}w$, and $l(u)>l(w)$.
\item {\em Local exchange}: $yuz \rightarrow ywz$ with $y,z \in X_V^*$,
$u,w \in X_v^*$ for some $v \in V$, $u=_{G_v}w$, and $l(u)=l(w)$.
\item {\em Shuffle}: $yuwz \rightarrow ywuz$ with $y,z \in X_V^*$,
$u \in X_v^*$ for some $v \in V$ and $w \in X_{v'}^*$ for some $v' \in \Ne(v)$.
\end{itemize}
Whenever a word $x \in X_V^*$ can be obtained
from another word $w \in X_V^*$ by a sequence
of local exchanges and shuffles, we write
$w \shuff x$, and whenever $x$ can be obtained from $w$
by a sequence of local reductions, local exchanges,
and shuffles, we write $w \lrles x$.

\begin{lemma}\emph{(\cite[Proposition~3.3]{CH14})}\label{lem:ops_geo}
Let $x$ be a geodesic in the graph product group $G_V$
with respect to the generating set $X_V$
and let $w$ be a word over $X_V$ satisfying
$w=_{G_V} x$.  Then $w \lrles x$.  Moreover, if
$w$ is also in $\geol(G_V,X_V)$, then $w \shuff x$.
\end{lemma}

The following (second) set of operations 
on words over $X_V$ preserve
the conjugacy class being represented.

\begin{itemize}
\item {\em Conjugate replacement}: $yuz \rightarrow ywz$ with $y,z \in X_V^*$,
$u,w \in X_v^*$ for some $v \in V$, $\supp(yz) \subseteq \Ne(v)$,
and $u \sim_{G_v}w$.
\item {\em Cyclic permutation}: $yu \rightarrow uy$ with $y \in X_V^*$
and $u \in X_v^*$ for some $v \in V$.
\end{itemize}
Whenever a word $x \in X_V^*$ can be obtained
from another word $w \in X_V^*$ by a sequence of 
local reductions and exchanges,
shuffles, conjugate replacements, and cyclic permutations,
we write $w \opsx x$.

In~\cite{ferov}, Ferov shows the following.

\begin{lemma}\label{lem:ops_cgeo_y}\cite[Lemma~3.12]{ferov}
If $x$ and $y$ are cyclic geodesics in the graph product
group $G_V$ with respect to the generating set $Y_V$,
and if $x \sim_{G_V} y$, then $x \opsx y$. 
Moreover, $\supp(x)=\supp(y)$ and $x \sim_{G_{\supp(x)}} y$.
\end{lemma}

In fact, again using the fact that over the generating
set $Y_v$ of a vertex group $G_v$ the 
geodesics and conjugacy geodesics are the words of
length 0 or 1, Ferov's proof only uses shuffles,
conjugate replacements consisting of replacing
a single letter in a vertex generating set $Y_v$ 
by another letter in that set, and cyclic permutations.
In the following, we extend Ferov's result
to the generating set $X_V$.

\begin{corollary}\label{cor:ops_cgeo_x}
Let $x$ be a cyclic geodesic in the graph product group $G_V$
with respect to the generating set $X_V$
and let $w$ be a word over $X_V$ satisfying
$w \sim_{G_V} x$.  Then $w \opsx x$.  Moreover, if
$w$ is also in $\geocpl(G_V,X_V)$, then 
$\supp(w) = \supp(x)$ and $w \sim_{G_{\supp(x)}} x$.
\end{corollary}

\begin{proof}
Starting from the word $w$, by repeatedly performing
local reductions and exchanges, shuffles, and
cyclic permutations,
after a finite number of steps
we must obtain a word $w_1$ for which no  
local reductions can occur in any further sequence.
Then Lemma~\ref{lem:ops_geo} shows that
the word $w_1 \in \geocpl(G_V,X_V)$.

Among all of the (finitely many) words that can be obtained
from $w_1$ by shuffles, let $w'$ be a word
with the minimum possible number of syllables
(where $w'$ is chosen to be $w_1$ if $w_1$ already
realizes the minimum).
Cyclically permute $w'$ by a single letter, and 
repeat the syllable minimization process
by shuffles.  Repeat this process until a
word $w_2$ is obtained for which no 
cyclic permutations of $w_2$ allow
shuffles that decrease the number of syllables.

We claim that $\xtoy(w_2) \in \geocpl(G_V,Y_V)$.
To show this, suppose instead that $\xtoy(w_2) \notin \geocpl(G_V,Y_V)$,
and write $w_2=u_1 \cdots u_n$ where the $u_i$
are the syllables of $w_2$.  For each
$1 \le i \le n$ let $v_i$ be the vertex
for which $u_i \in X_{v_i}^+$ and let
$g_i$ be the element of $G_{v_i} \setminus \{\groupid\}$
represented by $u_i$.  Then $\xtoy(w_2)=g_1 \cdots g_n$,
and there is an index $j$ such that 
$g_{j+1} \cdots g_ng_1 \cdots g_j \notin \geol(G_V,Y_V)$.
Applying Lemma~\ref{lem:ops_geo},
the word $g_{j+1} \cdots g_ng_1 \cdots g_j$ admits
a finite sequence of local shuffles leading to a local
reduction.  However, the corresponding sequence of
shuffles of the cyclic permutation $u_{j+1} \cdots u_nu_1 \cdots u_j$
of $w_2$ leads to a word with fewer syllables,
giving the required contradiction and proving the claim.

Similarly, there is a sequence of shuffles and
cyclic permutations from $x$ to another
word $x_2 \in \geocpl(G_V,X_V)$ satisfying 
$\xtoy(x_2) \in \geocpl(G_V,Y_V)$.  Now Lemma~\ref{lem:ops_cgeo_y} 
says that $\xtoy(w_2) \opsx \xtoy(x_2)$.

Construct a sequence of operations beginning 
from the word $w_2$ that follows the pattern
of the sequence $\xtoy(w_2) \opsx \xtoy(x_2)$,
in which each shuffle of the form
$\xtoy(y)\xtoy(p)\xtoy(q)\xtoy(z) \rightarrow \xtoy(y)\xtoy(q)\xtoy(p)\xtoy(z)$
of letters in $Y_V$ is
replaced by a shuffle of the corresponding syllables
$ypqz \rightarrow yqpz$ in $X_V^+$, 
each cyclic permutation $\xtoy(y)\xtoy(a) \rightarrow \xtoy(a)\xtoy(y)$
by a letter $\xtoy(a)$ in a vertex group generating set $Y_v$ is replaced by a 
cyclic permutation $ya \rightarrow ay$ by the corresponding syllable 
$a$ in $X_v^+$, and each conjugate replacement 
$\xtoy(y)\xtoy(p)\xtoy(z) \rightarrow \xtoy(y)\tilde q\xtoy(z)$
of a letter $\xtoy(p)$ in 
a set $Y_v$ is replaced by conjugate replacement $ypz \rightarrow yqz$
of the corresponding syllable $p$ in $X_v^+$
by any geodesic $q \in \geol(G_v,X_v)$ satisfying $q =_{G_V} \tilde q$.
Let $w_3$ be the word obtained from $w_2$ via this
sequence of operations on words.

Now $\xtoy(w_3)=\xtoy(x_2)$, and each syllable of 
$w_3$ and of $x_2$ is geodesic.  Hence there is
a sequence of local exchanges from $w_3$ to $x_2$.
 
Combining all of the sequences of operations above
shows that $w \opsx x$.  Moreover, if 
$w \in \geocpl(G_V,X_V)$, then we can take $w=w_1$.
Since none of the operations in the sequence from
$w=w_1$ to $x$ involve local reductions, and
the conjugate replacements in the sequence must
replace a word by another nonempty word over the same
vertex group generating set, 
these operations do not alter the support,
and moreover only involve conjugation by 
elements of $G_V$ whose support is in $\supp(w)$.
\end{proof}


\subsection{Shortlex and conjugacy representatives}\label{sub:slconjrep}

We now have the tools to show
that the results of Corollary~\ref{cor:geosubset}
hold for the shortlex and conjugacy shortlex languages as well.
A total ordering $<_V$ of the generating set $X_V$ of $G_V$
is called \emph{compatible with}
a total ordering $\ll$ of the vertex set $V$
of $\Gamma$ if 
 for each vertex $v \in V$ there is a total ordering
$<_v$ of the $X_v$ such that for all $a,b \in X_V$
we have $a<b$ if and only if either $\supp(a) \ll \supp(b)$
or $\supp(a) = \supp(b)$ and $a <_{\supp(a)} b$.

\begin{proposition}\label{prop:slsubset}
Let $G_V$ be a graph product group with generating set $X_V$,
let $V'$ be any subset of $V$.  
Let $\slex$ be a shortlex ordering on $X_V^*$ induced by an ordering
compatible with a total ordering $\ll$ on $V$, and let the shortlex
ordering on $X_{V'}^*$ be the restriction of the ordering $\slex$.
Then
\begin{eqnarray*}
\sphl(G_V,X_V) \cap X_{V'}^* &=& \sphl(G_{V'},X_{V'}) \text{ and}\\
\sphcl(G_V,X_V) \cap X_{V'}^* &=& \sphcl(G_{V'},X_{V'}).
\end{eqnarray*}
\end{proposition}

\begin{proof}
Suppose first that $w$ is a word in
$\sphl(G_V,X_V) \cap X_{V'}^*$.  Then no
shortlex smaller word over $X_V$ represents the
same element of $G_V$, and so no shortlex smaller
word over the subset $X_{V'}$ represents the same element
of the subgroup $G_{V'}$; hence $w \in \sphl(G_{V'},X_{V'})$.

On the other hand, if $w \in \sphl(G_{V'},X_{V'})$,
then Corollary~\ref{cor:geosubset} says that 
$w \in \geol(G_{V},X_{V}) \cap X_{V'}^*$.  
Then Lemma~\ref{lem:ops_geo} says that there is
a sequence of operations $w \shuff x$ (in the group
$G_V$ over the generating set $X_V$) from $w$ to
the shortlex least word $x$ over $X_V$ representing
the same element of $G_V$ as $w$.  Since all
of these operations also apply to the group
$G_{V'}$ over the generating set $X_{V'}$, then
$x \in \sphl(G_V,X_V) \cap X_{V'}^*$.  Moreover,
since $w$ and $x$ are both shortlex least
representatives in $X_{V'}^*$ of the same element of $G_{V'}$,
then $w=x$, completing the proof of the first equality
in Proposition~\ref{prop:slsubset}.

Next note that if $w$ is a word in
$\sphcl(G_V,X_V) \cap X_{V'}^*$, then
for all $g \in G_V$ we have
$w \slexeq x_g$ for the shortlex least representative
$x_g$ over $X_V$ of the element $gwg^{-1} \in G_V$.
In the first part of this proof, we show that
for all $g \in G_{V'}$, the word
$x_g$ is also the  shortlex least representative
over $X_{V'}$ of the element $gwg^{-1} \in G_{V'}$.
Hence $w \in \sphcl(G_{V'},X_{V'})$.

Finally suppose that $w \in \sphcl(G_{V'},X_{V'})$,
and let $x$ be the element of $\sphcl(G_V,X_V)$ 
satisfying $w \sim_{G_V} x$.  
Then Corollary~\ref{cor:ops_cgeo_x} says that
$w \opsx x$.  Again all
of these operations also apply to the group
$G_{V'}$ over the generating set $X_{V'}$, and so
$x \in \sphcl(G_V,X_V) \cap X_{V'}^*$.  
Now $w$ and $x$ are both shortlex least representatives in $X_{V'}^*$
of the same conjugacy class of $G_{V'}$, and so $w=x$.
\end{proof}

The following is useful for characterizing the shortlex 
least representatives of the elements of the graph product $G_V$,
and in particular shows that shortlex normal forms have 
geodesic images under $\xtoy$.

\begin{lemma}\label{lem:slxtoy}
Let $\slex$ be a shortlex ordering on words over the generating set $X_V$
of the graph product group $G_V$ induced by an ordering
compatible with a total ordering $\ll$ on $V$, let
$\ll_{sl}$ be a shortlex ordering on $Y_V^*$ compatible with $\ll$, and
let $u \in X_V^*$.  Then $u \in \sphl(G_V,X_V)$ if and only if
[$\xtoy(u) \in \sphl(G_V,Y_V)$ and each syllable of $u$
is in $\sphl(G_v,X_v)$ for some $v \in V$].
\end{lemma}

\begin{proof}
Suppose first that $u  \in \sphl(G_V,X_V)$. 
Lemma~\ref{lem:gpgeocgeo} shows that
$u \in \cap_{v \in V}~ 
  (\pi^X_v)^{-1}(\geol(G_v,X_v)(\$ \geol(G_v,X_v))^*)$, and since
any two $X_v$ letters of $u$ whose images under $\pi^X_v$ are
consecutive must also be consecutive in the shortlex normal form $u$, then
$u \in \cap_{v \in V}~ 
  (\pi^X_v)^{-1}(\sphl(G_v,X_v)(\$ \sphl(G_v,X_v))^*)$.
Moreover, if $\xtoy(u)$ is not geodesic, then there exist
two nonadjacent letters of $\xtoy(u)$ in the same subset $X_v$ (for some $v$)
that can be shuffled together so that
a local reduction can be applied;
hence the corresponding two syllables
of $u$ can be shuffled together, and so $u$ is not a shortlex least
representative of an element of $G_V$.  Similarly if $\xtoy(u)$
is geodesic but not in $\sphl(G_V,Y_V)$, then Lemma~\ref{lem:ops_geo}
says that there is a sequence of shuffles (since local exchanges cannot alter
an element of $\geol(G_V,Y_V)$) from $\xtoy(u)$ to
its shortlex normal form in $\sphl(G_V,Y_V)$.  Applying the
same shuffles to the corresponding syllables of $u$ results in a word
over $X_V$ that is smaller in the order $\slex$, contradicting 
that $u \in \sphl(G_V,X_V)$.  Hence $\xtoy(u) \in \sphl(G_V,Y_V)$.

Next suppose instead that $\xtoy(u) \in \sphl(G_V,Y_V)$
and each syllable of $u$
is in $\sphl(G_v,X_v)$ for some $v \in V$.  
Lemma~\ref{lem:gpgeocgeo} says that for each vertex $v \in V$, 
$\pi^Y_v(\xtoy(u)) \in (Y_v \cup \{\emptyword\})(\$(Y_v \cup \{\emptyword\}))^*$,
and hence no two distinct syllables of $u$ with support $v$
can be shuffled to be adjacent.  Thus each
$\pi^X_v(u)$ has the form $u_1 \$ u_2 \cdots \$u_n$ for some
$n \ge 1$, where each $u_i$ is a syllable of $u$, and so
$\pi^X_v(u) \in \geol(G_v,X_v) (\$\geol(G_v,X_v))^*$.
Now Lemma~\ref{lem:gpgeocgeo} says that $u \in \geol(G_V,X_V)$.

Let $u' \in \sphl(G_V,X_V)$ satisfy $u'=_{G_V}u$; that is, let $u'$ be
the shortlex normal form for the group element represented by $u$.
By the first part of this proof, we have $\xtoy(u') \in \sphl(G_V,Y_V)$,
and so $\xtoy(u)=\xtoy(u')$.
Lemma~\ref{lem:ops_geo} says that $u \shuff u'$.
For each $v \in V$, shuffles applied to $u$
cannot change the image of the homomorphism $\pi^X_v$, and so
$\pi^X_v(u)=\pi^X_v(u')$.  Moreover, since $\xtoy(u)=\xtoy(u')$ is a geodesic over $Y_V$,
no sequence of shuffles applied to $u$ or $u'$ can result in fewer syllables.
Hence the syllables of both $u$ and $u'$ are the same, the
syllables lie in $\sphl(G_V,X_V)$,
and they occur in the same order.  Therefore $u=u'$, and so $u \in \sphl(G_V,X_V)$.
\end{proof}


\subsection{Decomposition of graph products into amalgamated products, admissible transversals, and growth formulas}\label{subsec:amalg}


The computation of the standard growth series of a graph product
by Chiswell in~\cite{C94} involves decomposing the graph product into an amalgamated product, and applying the concept of ``admissible subgroups''. 
In this section we give a brief summary of these results,
and describe a language representing an admissible transversal for
a subgraph product in a graph product.

Each
graph product over a graph with more than one vertex can be decomposed
as an amalgamated product of graph products of groups over the graph product of an appropriate centralizing set.

\begin{lemma}[\cite{G90},\cite{C94}]\label{lem:decomposition_throw_a_vertex}
 Let $G_V$ be a graph product of groups, and let $v\in V$.
Using the inclusion maps from $G_{\Ne(v)}$ into
both $G_{V \setminus \{v\}}$ and
$G_{\Ne(v) \cup \{v\}}=G_{\Ne(v)}\times G_{v}$,
the group $G_V$ can be decomposed as the amalgamated product 
 \[
  G_V = G_{V\setminus\{v\}} ~\ast_{G_{\Ne(v)}}~ (G_{\Ne(v)}\times G_{v}) 
 \]
\end{lemma}

\begin{definition}[\cite{Al91},\cite{Le91}]\label{def:adm}
 Let $G$ be a group, $H$ a subgroup of $G$, $X$ an inverse-closed generating set of $G$ and $Y$ an inverse-closed generating set of $H$. 
 The group $H$ is {\em admissible in} $G$ 
{\em with respect to the pair} $(X,Y)$ if $Y\subset X$ and 
there exists a right transversal $U_{H\backslash G} \subseteq G$ 
for $H$ in $G$ such that whenever $g=hu$ with $g\in G$, $h\in H$ and $u\in U_{H \backslash G}$, then $\eltl{g}_X=\eltl{h}_Y+\eltl{u}_X$. We assume that the transversal contains the identity as representative of $H$, and say that $U_{H \backslash G}$ is an {\em admissible} right transversal of $H$ in $G$
with respect to $(X,Y)$.
\end{definition}

\begin{remark}\label{rmk:admgrowth}
For an admissible subgroup $H=\langle Y \rangle$
of $G=\langle X \rangle$ with admissible transversal
$U_{H \backslash G}$, the \wsphs\ satisfy
the relation $\sgs_{(G,X)}=\sgs_{(H,Y)}\sgs_{(U_{H \backslash G},X)}$, where $\sgs_{(U_{H \backslash G},X)}$ denotes the growth series of the elements of the transversal $U_{H \backslash G}$ with respect to $X$.
\end{remark}

The next lemma shows the relationship between
the spherical growth series of a free product
of groups amalgamated along a common admissible
subgroup, and the spherical growth series of
the factor and amalgamating subgroups.

\begin{lemma}[\cite{Al91},\cite{Le91}]\label{lem:sgs_amalamed_prod_admi}
 Let $G$, $K$ be groups and let $H$ be a subgroup of both $G$ and $K$. Let $X$, $Y$ and $Z$ be inverse-closed generating sets of $G$, $H$ and $K$, respectively.
 Suppose that $H$ is admissible in both $G$ and $K$ with respect to the pairs $(X,Y)$ and $(Z,Y)$, respectively. Let $A$ be the amalgamated product $A:=G\ast_{H}K$ and let $W:=X\cup Z$.
Then
 \[
  \frac{1}{\sgs_{(A,W)}}=\frac{1}{\sgs_{(G,X)}}+\frac{1}{\sgs_{(K,Z)}}-\frac{1}{\sgs_{(H,Y)}}.
 \]
\end{lemma}

\begin{remark}\label{rmk:subgpisadmissible}
Given groups $G_i=\langle X_i \rangle$ for $i=1,2$,
it follows directly from Definition~\ref{def:adm} that 
$G_1$ is admissible in the
direct product group $G_1 \times G_2$ 
with respect to the pair of generating sets
$(X_1 \cup X_2,X_1)$, with admissible transversal
$\{\groupid\} \times G_2$.
\end{remark}

The following formula  for computing the
spherical growth series of a graph product from 
spherical growth series of subgraph products 
is an immediate consequence of
Lemmas~\ref{lem:sgs_amalamed_prod_admi} and~\ref{lem:decomposition_throw_a_vertex} and
Remarks~\ref{rmk:admgrowth} and~\ref{rmk:subgpisadmissible};
this formula was obtained by 
Chiswell in~\cite[Proof of Proposition~1]{C94}.
This recursive formula is the analog
for spherical growth series of our formula in
Theorem~\ref{thmx:the_formula_for_cgs}
for spherical conjugacy growth series.

\begin{corollary}\label{cor:recursionforgrowth}
Let $G_V$ be a graph product group over a graph
with vertex set $V$ and let $v\in V$.
For each $v' \in V$ let $X_{v'}$ be an inverse-closed generating
set for the vertex group $G_{v'}$, and 
for each $S \subseteq V$ let $\sgs_S$ be the spherical growth
series for the subgraph product $G_S$ 
on the subgraph induced by $S$, over the generating
set $X_S = \cup_{v' \in S}X_{v'}$. Then
\[
 \sgs_{V}=\frac{\sgs_{\Ne(v)} \sgs_{V\setminus\{v\}}\sgs_{\{v\}}}{\sgs_{\Ne(v)}\sgs_{\{v\}}+\sgs_{V\setminus\{v\}}-\sgs_{V\setminus\{v\}}\sgs_{\{v\}}}.
\]
\end{corollary}

Recall that
if each vertex group $G_v$ of a graph product
on a graph with vertex set $V$ 
has an inverse-closed generating set $X_v$,
then for each $V' \subseteq V$, the subgraph product
$G_{V'}$ has generating set $X_{V'} := \cup_{v \in W} X_v$.
Using these generating sets,
any subgraph product $G_{V'}$ is an admissible
subgroup in a graph product $G_V$ with respect to the pair $(X_V,X_{V'})$
(see~\cite{C94},\cite[Proposition~14.4]{M12}).
In the following we provide a set of representatives for 
a specific admissible transversal for a subgraph product
in a graph product, which we will use in our proofs in
Section~\ref{sec:conj_growth_series}.

\begin{lemma}\label{lem:subgpadm} 
Let $G_V$ be a graph product with vertex set $V$,
for each $v \in V$ let $X_v$ be an inverse-closed generating set for $G_v$,
and let $V' \subseteq V$.  Let $\slex$ be a shortlex ordering on 
$X_V^*$ compatible with a total ordering $\ll$ on $V$
satisfying $a \ll b$ for all $a \in V'$ and $b \in V \setminus V'$.
Then the set of words
\[
\widehat U_{G_{V'} \backslash G_V} := 
\{\emptyword\} \cup (\sphl(G_V,X_V) \cap (X_{V \setminus V'}X_V^*))
\]
is a set of unique representatives of  an
admissible transversal $U_{G_{V'} \backslash G_V}$ for the
subgraph product group $G_{V'}$ in $G_{V}$ 
with respect to the pair $(X_V,X_{V'})$.
\end{lemma}

\begin{proof}
Let $g$ be any element of $G_V$, and let
$y$ be the shortlex normal form of $g$.
Then there is a factorization $y=y_1y_2$ where
$y_1$ is the longest prefix of $y$ lying in
$X_{V'}^*$.  Now either $y_2=\emptyword$,
or else the first letter of $y_2$ lies in $X_{V \setminus V'}$.
Since every subword of a shortlex normal form
is again a shortlex least representative of
a group element, then $y_2 \in \widehat U_{G_{V'} \backslash G_V}$.
Hence $\widehat U_{G_{V'} \backslash G_V}$ contains
representatives of elements in every coset.

Next suppose that $w$ is any word in $\sphl(G_{V'},X_{V'})$
and $u \in \widehat U_{G_{V'} \backslash G_V}$; in this
paragraph we show that $wu$ is a geodesic
in $G_V$ over $X_V$ using Lemma~\ref{lem:gpgeocgeo}.
Given $v \in V \setminus V'$, the image of $wu$ under the 
homomorphism $\pi_v$ associated to $v$ satisfies
$\pi_v(wu)=\pi_v(w)\pi_v(u)$,
where $\pi_v(w) \in \$^*$ and
$\pi_v(u) \in \geol(G_v,X_v)(\$\geol(G_v,X_v))^*$
by Lemma~\ref{lem:gpgeocgeo} since $u$ is a shortlex
normal form and hence a geodesic.
On the other hand, given $v \in V'$, we have
$\pi_v(wu)=\pi_v(w)\pi_v(u)$
where $\pi_v(w) \in \geol(G_v,X_v)(\$\geol(G_v,X_v))^*$
since $w$ is a geodesic.
Either $\supp(u) \subseteq \Ne(v)$,
in which case  
$\pi_v(u)=\emptyword$,
or else we can write the shortlex normal form
$u=u_1cu_2$ for some
$u_1 \in X_{\Ne(v)}^*$ and $c \in X_{V \setminus \Ne(v)}$.  
In the latter case, since the first letter $b$ of $u$ lies 
in $X_{V \setminus V'}$, then $b>a$ for all $a \in X_v$. 
Since $u$ is the shortlex
least representative of a group element, 
we have $u=u_1cu_2 \neq cu_1u_2$, and 
consequently $c \notin X_v$.  Therefore
the first letter of $\pi_v(u)$is $\$$, and
in this case the image
of the geodesic $u$ satisfies 
$\pi_v(u) \in (\$\geol(G_v,X_v))^*$.
Hence in all cases we have
$\pi_v(wu) \in \geol(G_v,X_v)(\$\geol(G_v,X_v))^*$.
Then Lemma~\ref{lem:gpgeocgeo} shows that $wu$ is geodesic.

Finally suppose that $wu=_{G_V} w'u'$ for some
$w,w' \in \sphl(G_{V'},X_{V'})$ and
$u,u' \in \widehat U_{G_{V'} \backslash G_V}$.
Then $u=_{G_V} w''u'$ where $w''$ is the
element of $\sphl(G_{V'},X_{V'})$ representing $w^{-1}w'$.
By the preceding paragraph, then $u$ and $w''u$
are geodesics representing the same element of $G_V$.
Now Lemma~\ref{lem:ops_geo} shows that $u \shuff w''u'$;
that is, $w''u'$ can be obtained from $u$
by a sequence of local exchanges and shuffles.
Suppose that $w'' \neq \emptyword$, and let
$v \in V'$ be the support of the first letter $a$ of $w''$.
Then the first letter of $\pi_v(w''u)$ is $a$,
and the argument in the previous paragraph
shows that either $\pi_v(u)=\emptyword$
or the first letter of $\pi_v(u)$ is $\$$.
Note that the shuffle operation does not
change the image of any word under the
$\pi_v$ homomorphism, and the only change
possible under a local exchange is the
replacement of one subword of $X_v^*$
by another of the same length. 
Hence the word $w''u$ cannot be obtained
from $u$; this contradiction shows that
$w''=\emptyword$.  Therefore $w=_{G_V} w'$
(and so $w=w'$).
Consequently we also have $u=_{G_V} u'$, and
since $u,u'$ are shortlex normal forms,
$u=u'$ as well.  Thus each coset has only
one representative in
$\widehat U_{G_{V'} \backslash G_V}$,
completing the proof that this is a set of 
unique representatives of an admissible transversal.
\end{proof}


\section{The conjugacy growth series of a graph product}\label{sec:conj_growth_series}


In this section we will first determine
a set of conjugacy geodesic representatives of the conjugacy classes
of a graph product, in Section~\ref{subsec:conjgeoreps}.  
Then in Section~\ref{subsec:equalrates} we establish preservation
of equality of standard and 
conjugacy growth rates by a graph product, and 
in Section~\ref{ss_conj_series} 
we derive the recursive formula for the spherical 
conjugacy growth series.


\subsection{Conjugacy geodesic representatives of conjugacy classes}\label{subsec:conjgeoreps}

In Proposition~\ref{prop:UC_geo} we apply the characterisation of geodesics and
conjugacy geodesics in graph products
from Lemma~\ref{lem:gpgeocgeo}
to the amalgamated product decomposition
of Lemma~\ref{lem:decomposition_throw_a_vertex}.

Throughout Section~\ref{subsec:conjgeoreps} we will assume the following:

\textbf{Hypothesis A}: Let $G_V$ be a graph product group, with generating set $X_V$, 
and let $v \in V$
be a vertex for which $\{v\} \cup \Ne(v)\subsetneq V$.
Let $\slex$ be a shortlex ordering on $X_{V}^*$ 
that is compatible with a total ordering $\ll$ on $V$
satisfying $x \ll y$ for all $x \in \Ne(v)$ and 
$y \in V \setminus (\Ne(v) \cup \{v\})$,
and let $\widehat U:=\widehat U_{G_{\Ne(v)} \backslash G_{V \setminus \{v\}}}$
be the admissible transversal set of representatives for $G_{\Ne(v)}$
in $G_{V \setminus \{v\}}$ with respect to
$(X_{V \setminus \{v\}},X_{\Ne(v)})$ from Lemma~\ref{lem:subgpadm}.

\begin{proposition}\label{prop:UC_geo}
Let $G_V$ and $v \in V$ satisfy Hypothesis A.
Suppose that 
$u_i \in \widehat U
\setminus \{\emptyword\}$ and 
$c_i \in \geol(G_v,X_v) \setminus \{\emptyword\}$ for all $i$,  
and that $b \in \geol(G_{\Ne(v)},X_{\Ne(v)})$,
$\tilde b \in \geocl(G_{\Ne(v)},X_{\Ne(v)})$, and
$\supp(\tilde b) \subseteq \Ne(\cup_{i=1}^n \supp(u_i))$.  Then:
\begin{itemize}
\item[(1)] 
The words $bu_1c_1 \cdots u_nc_n$
and $bc_0u_1c_1 \cdots u_nc_n$ are geodesics
in $G_V$ over $X_V$.
\item[(2)] The word $\tilde bu_1c_1 \cdots u_nc_n$ is a conjugacy geodesic
in $G_V$ over $X_V$.
\end{itemize}
\end{proposition}

\begin{proof}
Let $w=u_1c_1 \cdots u_nc_n$.
We consider the images of the words $bw$, $bc_0w$, and
$\tilde bw$ under
the $\pi_{v'} = \pi^{X_V}_{v'}$ maps, for $v' \in V$, in turn.

In the case that $v'=v$, note that  
$$
\pi_v(b)=\pi_v(\tilde b)=\emptyword, 
\hspace{.4in}
\pi_v(c_i)=c_i \in \geol(G_v,X_v),
\hspace{.4in}
\pi_v(u_i) \in \$^+,
$$ 
where the latter containment
follows from the fact that the first letter of $u_i$ lies 
in $X_{V \setminus (\Ne(v) \cup \{v\})}$,
and hence the word $\pi_v(u_i)$ is nonempty.
Then $\pi_v(bw)=\pi_v(\tilde bw)=\pi_v(w)=\$^{i_1}c_1 \cdots \$^{i_n}c_n
\in (\$\geol(G_v,X_v))^*$
for some natural numbers $i_1,...,i_n$, and
$\pi_v(bc_0w) \in  \geol(G_v,X_v)(\$\geol(G_v,X_v))^*$.

Next consider the case that 
$v' \in V \setminus (\Ne(v) \cup \{v\})$.
Applying Lemma~\ref{lem:gpgeocgeo} to $u_i$ since
$u_i$ is a geodesic in $G_V$ over $X_V$ (from Corollary~\ref{cor:geosubset}), we have
$$
\pi_{v'}(b), \pi_{v'}(\tilde b) \in \$^*,
\hspace{.1in}
\pi_{v'}(c_i) \in \$^+,
\hspace{.1in}
\pi_{v'}(u_i) \in \geol(G_{v'},X_{v'})(\$\geol(G_{v'},X_{v'}))^*.
$$
Hence in this case 
$\pi_{v'}(bw), \pi_{v'}(bc_0w), \pi_{v'}(\tilde bw)
\in (\geol(G_{v'},X_{v'})\$)^*$.

Finally suppose that $v' \in \Ne(v)$.
Let $a_i$ be the first letter of the word $u_i$;
then $\supp(a_i) \in  V \setminus (\Ne(v) \cup \{v\})$.
If the word $\pi_{v'}(u_i)$ were to start with a letter $a$ in
$X_{v'}$, then $u_i$ can be shuffled to a word beginning with $a$,
contradicting the fact that $u_i \in \widehat U$ is a shortlex normal form
and $a \slex a_i$ in the shortlex ordering (compatible with $\ll$). 
Hence $\pi_{v'}(u_i)$ is either $\lambda$ or starts with $\$$.
In this case (applying Lemma~\ref{lem:gpgeocgeo} and Corollary~\ref{cor:geosubset} again) we have $\pi_{v'}(c_i) = \emptyword,$
$$
\pi_{v'}(b) \in \geol(G_{v'},X_{v'})(\$\geol(G_{v'},X_{v'}))^*,
\hspace{.1in}
\pi_{v'}(u_i) \in (\$\geol(G_{v'},X_{v'}))^*.
$$
Moreover, either $v' \notin \supp(\tilde b)$ and $\pi_{v'}(\tilde b) \in \$^*$,
or else $v' \in \supp(\tilde b) \subseteq \Ne(\cup_{i=1}^n \supp(u_i))$
and hence (by Lemma~\ref{lem:gpgeocgeo} and Corollary~\ref{cor:geosubset})
$\pi_{v'}(\tilde b) \in \geocl(G_{v'},X_{v'}) \cup \cycperm((\$ \geol(G_{v'},X_{v'}))^+)$
 and
$\pi_{v'}(u_i)=\emptyword$ for all $i$.

Thus for all $v' \in V$ we have
$\pi_{v'}(w) \in (\$\geol(G_{v'},X_{v'}))^*$,
$\pi_{v'}(bw),\pi_{v'}(bc_0w) \in \geol(G_{v'},X_{v'})(\$\geol(G_{v'},X_{v'}))^*$,
and $$\pi_{v'}(\tilde b w) \in \geocl(G_{v'},X_{v'}) \cup \cycperm((\$ \geol(G_{v'},X_{v'}))^+).$$
Lemma~\ref{lem:gpgeocgeo} then completes the proof
of~(1) and~(2).
\end{proof}

A \emph{piecewise subword} of a word $b \in X_V^*$
is a word over $X_V$ of the form $b_1 \cdots b_k$ such that
$b=d_0b_1d_1 \cdots b_kd_k$ for some words $d_0,...,d_k \in X_V^*$.
A piecewise subword $b'$ of $b$ is \emph{proper} if
$b' \neq b$.
In the following lemma, we show that multiplying the
geodesics in Proposition~\ref{prop:UC_geo} on the right
by a word $b$ over $X_{\Ne(v)}$ yields an element represented by
another such geodesic in which a piecewise subword of $b$ occurs on the left.

\begin{lemma}\label{lem:UCb}
Let $G_V$ and $v \in V$ satisfy Hypothesis A.
Suppose that $u_1 \in \widehat U$,
$u_i \in \widehat U \setminus \{\emptyword\}$ for all $i>1$,
$c_n \in \sphl(G_v,X_v)$, 
$c_i \in \sphl(G_v,X_v) \setminus \{\emptyword\}$ for all $i<n$, 
and $b \in \sphl(G_{\Ne(v)},X_{\Ne(v)})$.
Then: 
\begin{itemize}
\item[(1)] 
$u_1 c_1 \cdots u_nc_n b$ is equal in $G_V$
to a word of the form $b' u_1'c_1 \cdots u_n'c_n$
satisfying $u_1' \in \widehat U$, with $u_1' = \emptyword$ if and only if $u_1=\emptyword$,
$u_i' \in \widehat U \setminus \{\emptyword\}$ for all $i>1$, 
and $b'$ is a piecewise subword of $b$. 
Moreover, if $\supp(b) \not\subseteq \Ne(\cup_{i=1}^n \supp(u_i))$, then
$b'$ is a proper piecewise subword of $b$.

\item[(2)] 
$b u_1 c_1 \cdots u_nc_n$ can be conjugated by an element of $G_{\supp(b)}$
to an element of $G_V$ represented by a word of the form $\tilde bu_1'c_1 \cdots u_n'c_n$
satisfying $u_1' \in \widehat U$ and $u_1' = \emptyword$ if and only if $u_1=\emptyword$,
$u_i' \in \widehat U \setminus \{\emptyword\}$ for all $i>1$,
$\tilde b \in \sphcl(G_{\Ne(v)},X_{\Ne(v)})$,
and $\supp(\tilde b) \subseteq \Ne(\cup_{i=1}^n \supp(u_i'))$.
\end{itemize} 
\end{lemma}

\begin{proof}
We begin by proving item (1) in the special case that $n=1$,
$u_1 \in \widehat U \setminus \{\emptyword\}$,
$c_1=\emptyword$, and $b \in X_{v'}$ is a single letter, with
$v' \in \Ne(v)$.  

{\it Case 1.  Suppose that $\supp(b) \subseteq \Ne(\supp(u_1))$.} 
Then $u_1b =_{G_V} bu_1$, which has the required form.

{\it Case 2.  Suppose that
$\supp(b) \not \subseteq \Ne(\supp(u_1))$.} 
Then we can write $u_1=xyz$ where $z$ is the maximal suffix
of $u_1$ satisfying $\supp(b) \subseteq \Ne(\supp(z))$,
and $y$ is a syllable of $u_1$.

{\it Case 2a.  Suppose that $\supp(y) \neq \supp(b)$.} 
Then $b$ is a syllable of the word $ub$ and $\xtoy(u_1b)=\xtoy(u_1)\xtoy(b)$.
Since $u_1 \in \sphl(G_{V \setminus \{v\}},X_{V \setminus \{v\}})$,
Lemma~\ref{lem:slxtoy} says that 
$\xtoy(u_1) \in \sphl(G_{V \setminus \{v\}},Y_{V \setminus \{v\}})$
and each syllable of $u_1$ is in $\sphl(G_{\hat v},X_{\hat v})$ for
some $\hat v$.
For all $\hat v \in V$ we have 
$\pi^Y_{\hat v}(\xtoy(u_1b))=\pi^Y_{\hat v}(u_1)\pi^Y_{\hat v}(b)$, 
and Lemma~\ref{lem:gpgeocgeo} says that
$\pi^Y_{\hat v}(\xtoy(u_1)) \in 
  \geol(G_{\hat v},Y_{\hat v})(\$\geol(G_{\hat v},Y_{\hat v}))^*$. 

For each $\hat v \neq v'$, 
either $\pi^Y_{\hat v}(\xtoy(u_1b))=\pi^Y_{\hat v}(\xtoy(u_1))$,
or $\pi^Y_{\hat v}(\xtoy(u_1b))=\pi^Y_{\hat v}(\xtoy(u_1))\$$.
Also 
$\pi^Y_{v'}(\xtoy(u_1b))=\pi^Y_{v'}(\xtoy(x))\$\xtoy(b)$.
Hence $$\xtoy(u_1b) \in \cap_{\hat v \in V} 
(\pi^Y_{\hat v})^{-1} (\geol(G_{\hat v},Y_{\hat v})(\$\geol(G_{\hat v},Y_{\hat v}))^*),$$
and by Lemma~\ref{lem:gpgeocgeo} the word $\xtoy(u_1b)$
is a geodesic in $G_V$ over $Y_V$.
Now Lemma~\ref{lem:ops_geo} says that there is a sequence
of shuffles from $\xtoy(u_1b)$ to its shortlex normal form. 
Let $u' \in X_{V \setminus \{v\}}^*$ be the word obtained from
$u_1b$ by performing the same shuffles to the associated syllables of $u_1b$.
Then $\xtoy(u') \in \sphl(G_{V \setminus \{v\}},Y_{V \setminus \{v\}})$
and each syllable of $u'$ is (either $b$ or a syllable of $u_1$ and
hence) in the shortlex language of its
vertex group.  Now Lemma~\ref{lem:slxtoy} says that 
$u' \in \sphl(G_{V \setminus \{v\}},X_{V \setminus \{v\}})$.
Moreover,  since shuffles cannot alter the image of a word under a
$\pi^X_{\hat v}$ map, and since $\pi^X_{\hat v}(u_1b)$ is
either the empty word or starts with a $\$$ for every
$\hat v \in \Ne(v)$ (by definition of $\widehat U$ and the
choice of the ordering $\slex$ compatible with $\ll$),
the same is true for the shuffled word $u'$.  Hence
$u' \in X_{V \setminus (\{v\} \cup \Ne(v))}X_{V \setminus \{v\}}^*$,
and so $u' \in \widehat U$.  Therefore 
$u_1b =_{G_V} u'$ for a word $u' \in \widehat U$ in case~2a.

{\it Case 2b.  Suppose that $\supp(y) = \supp(b)$.} 
Let $y'$ be the shortlex normal form for $yb$.
Since $y$ and $z$ commute and $xyz$ is in shortlex
form, the rightmost syllable of $x$ and the leftmost
syllable of $z$ cannot have the same support, and so
(irrespective of whether or not $y'$ is the empty word)
the syllables of $xy'z$ are either $y'$ or syllables
of $x$ or of $z$ and hence are syllables of $u_1$.  Thus each syllable of $xy'z$ 
is in $\sphl(G_{\hat v},X_{\hat v})$ for some $\hat v$.
Following an arugment similar to that in Case~2a,
the word $\xtoy(u_1) \in \sphl(G_{V \setminus \{v\}},Y_{V \setminus \{v\}})$,
and the word $\xtoy(xy'z)$ is obtained
from $\xtoy(u_1)=\xtoy(xyz)$ either by a local
exchange of a letter $\xtoy(y)$ for a letter $\xtoy(y')$
if $y' \neq \emptyword$, in which case the word
$\xtoy(xy'z)$ is again in $\sphl(G_{V \setminus \{v\}},Y_{V \setminus \{v\}})$,
or else by removal of the letter $\xtoy(y)$, if $y' = \emptyword$.
In the latter situation, an argument similar to that in case~2a, 
using the maps $\pi^Y_{\hat v}$, can be used to show that 
the word $\xtoy(xy'z)=\xtoy(xz)$ is geodesic, and moreover
is in $\sphl(G_{V \setminus \{v\}},Y_{V \setminus \{v\}})$.
Hence Lemma~\ref{lem:slxtoy} shows that 
$xy'z \in \sphl(G_{V \setminus \{v\}},X_{V \setminus \{v\}})$.
Since the first letter $a$ of the word $u_1=xyz$ lies
in $X_{V \setminus (\{v\} \cup \Ne(v))}$, the
subword $x$ is nonempty and the first letter of
the word $u':=xyz$ is also $a$.  
Therefore 
$u_1b =_{G_V} u'$ for a word $u' \in \widehat U$ in case~2b also.

This completes the proof of the special case.  For
the general case of part (1), let $w=u_1 c_1 \cdots u_nc_n$
and write $b=b_1 \cdots b_m$ with each $b_i \in X_{\Ne(v)}$.
Starting with the word $wb$, shuffle $b_1$ to the left 
until either $b_1$ reaches the left side of the word, or
$b_1$ reaches a subword $u_j$ such that $\supp(b_1) \not\subseteq \supp(u_j)$,
in which case the special case above is applied to 
replace $u_j$ by another element of $\widehat U$.  
Iterating this for the letters $b_2$ through $b_m$ completes
the proof of (1).

Note that although cyclic conjugation of $bw$ to $wb$
and then applying the process from part (1) above results in a
word $b'\hat u_1c_1 \cdots \hat u_nc_n=_{G_V}wb$ with 
each $\hat u_i \in \widehat U$ and 
$b'$  
a piecewise
subword of $b$ that is potentially shorter than $b$, it is possible
that $\supp(b') \not \subseteq \Ne(\cup_{i=1}^n \supp(\hat u_i))$.

Iterate this process of cyclically conjugating the
maximal prefix in $X_{\Ne(v)}^*$ to the right side of the word
and applying the algorithm above.  Since the word length of
the prefix in $X_{\Ne(v)}^*$ can only strictly decrease finitely
many times, after finitely many steps,
the procedure must reach a word of the form
$ b''w'=b'' u_1'c_1 \cdots u_n'c_n$ such that the 
algorithm above applied to $w'b''$ results in
$b'' w'$; that is, 
$\supp(b'') \subseteq \Ne(\cup_{i=1}^n \supp(u_i')) \cap \supp(b)$. 
Finally, let $\tilde b \in \sphcl(G_{\supp(b)},X_{\supp(b)})$
be the shortlex least word representing an element of the
conjugacy class of $G_{\supp(b)}$ containing $b''$;
Corollary~\ref{cor:geosubset} shows that
$\tilde b \in \sphcl(G_{\Ne(v)},X_{\Ne(v)})$ as well.
Now there is an element $g \in G_{\supp(b)}$
such that $gb''g^{-1} =_{G_{\supp(b)}} \tilde b$,
and so $\tilde b w' =_{G_{V}} gb''w'g^{-1}$ is a 
conjugate of $wb$ by an element of $G_{\supp(b)}$ as well.
This completes the proof of (2).
\end{proof}

Following the notation in~\cite[Section~IV.2]{LS01}, a sequence
$a_1,...,a_n$ (with $n \ge 0$) of elements of the amalgamated product
$G=A*_C B$
 is {\em reduced} if each $a_i$ is in one of two subgroups
$A$ or $B$, successive 
$a_i$ are not in the same subgroup, 
if $n=1$ then $a_1 \neq \groupid$,
and if $n>1$ then no $a_i$ is in $C$.
This sequence is {\em cyclically reduced} if every cyclic
permutation of the sequence is reduced.  

In the following we apply
the normal form and conjugacy normal form theorems~\cite[Theorems~IV.2.6,IV.2.8]{LS01} for sequences in
free products with amalgamation to establish conjugacy representatives
for every conjugacy class of a graph product, and 
to determine when two of these conjugacy
geodesics represent the same conjugacy class.

\begin{proposition}\label{prop:UC_pearls}
Let $G_V$ and $v\in V$ satisfy Hypothesis A.
\begin{itemize}
\item[(1)]  
For each element $g \in G_V$ 
there exists a conjugacy geodesic $w \in X_V^*$ 
representing the conjugacy class $[g]_{\sim,{G_V}}$, with $w$ 
either of the form
\begin{itemize}\label{min_conj_elements}
\item[$\dagger$]  $w=\tilde{b}{u}_1 {c}_1\cdots {u}_n {c}_n$,
where $n>0$, 
${u}_i\in \widehat U\setminus \{\emptyword\}$, 
${c}_i \in \sphl(G_{v},X_{v})\setminus \{\emptyword\}$, 
$\tilde{b} \in \sphcl(G_{\Ne(v)}, X_{\Ne(v)})$, and 
$\supp(\tilde b) \subseteq \Ne(\cup_{i=1}^n \supp(u_i))$,
 \end{itemize}
or else of the form
\begin{itemize}\label{min_conj_elements_other}
\item[$\ddagger$] $w \in \sphcl(G_{V\setminus\{v\}},X_{V\setminus\{v\}})
\cup \sphcl(G_{\{v\} \cup \Ne(v)},X_{\{v\} \cup \Ne(v)}).$
\end{itemize}
\item[(2)] Two words
$w_1,w_2 \in X_V^*$ that are each of the form $\dagger$ or $\ddagger$ 
 represent conjugate elements of $G_V$ if and only if 
either $w_1=w_2$, or
the words can be written   
$w_1=\tilde b u_1c_1 \cdots u_nc_n$ and 
$w_2=\tilde b'u_1'c_1' \cdots u_{n'}'c_{n'}'$ in the form $\dagger$
such that
\begin{enumerate}
\item[(i)] $\tilde b = \tilde b'$ and $n=n'$, and
\item[(ii)] there is an index $j$ such that $u_i=u_{i+j}'$
and $c_i=c_{i+j}'$ for all $i$, where the indices
are considered modulo $n$.
\end{enumerate}
\end{itemize}
\end{proposition}

\begin{proof} 
Let $g$ be any element of $G_V$. Using Lemma~\ref{lem:decomposition_throw_a_vertex} 
and the normal form theorem for amalgamated products 
(see for example~\cite[Theorem~IV.2.6]{LS01}), the element $g$ is
represented by a word of the form
$x=\hat b \hat u_1 \hat c_1 \cdots \hat u_n \hat c_n$ for
some $n \ge 0$, $\hat u_i \in \widehat U$ for all $i$ with
$\hat u_i \neq \emptyword$ for all $i>1$,
$\hat c_i \in \sphl(G_v,X_v)$ for all $i$ with $\hat c_i \neq \emptyword$
for all $i<n$, and
$\hat b \in \sphl(G_{\Ne(v)},X_{\Ne(v)})$.
By Lemma~\ref{lem:UCb}(2), then $g$ is conjugate in $G_V$ to
another element $g'$ represented by a word of the form
$x':=\hat b'\hat u_1'\hat c_1 \cdots \hat u_n'\hat c_n$
satisfying $\hat u_1' \in \widehat U$, 
$\hat u_i' \in \widehat U \setminus \{\emptyword\}$ for $i>1$,
$\hat b' \in \sphcl(G_{\Ne(v)},X_{\Ne(v)})$,
and $\supp(\tilde b) \subseteq \Ne(\cup_{i=1}^n \supp(\hat u_i'))$.

If $n=0$ then $x'$ is of the form $\ddagger$.  Suppose instead
that $n>0$.

If both $\hat u_1'$ and $\hat c_n$ are not the empty word, then
$x'$ is in the form $\dagger$.
If both $\hat u_1'$ and $\hat c_n$ are the empty word, 
then $g$ is conjugate to the element of $G_V$ represented by 
$\hat b'\hat u_n'\hat c_1 \cdots \hat u_{n-1}'\hat c_{n-1}$, 
which is of the form $\dagger$ (or $\ddagger$ if $n=1$).

On the other hand, if exactly one of $\hat u_1',\hat c_n$ is equal to $\emptyword$,
then using the fact that 
$g$ is also conjugate to
$g''=_{G_V} \hat b'\hat u_n'\hat c_n\hat u_1'\hat c_1 \cdots \hat u_{n-1}'\hat c_{n-1}$,
we can replace any consecutive $\hat c_n\hat c_1$ 
by the shortlex least representative
of this element in $G_v$ over $X_v$, and
we can replace any consecutive $\hat b'\hat u_n'\hat u_1'$ 
(or $\hat b'\hat u_n'\hat u_2'$ if $\hat u_1'=\emptyword$ and 
$\hat c_n\hat c_1=_{G_v} \emptyword$) by 
$du''$ for some 
$d \in \geol(G_{\Ne(v)},X_{\Ne(v)})$ and $u'' \in \widehat U$,
since $\widehat U$ is a set of representatives of a transversal.

We repeat this process iteratively; that is, at each step
we conjugate by a word over $X_{\Ne(v)}$
in order to apply Lemma~\ref{lem:UCb}(2), and then (cyclically)
conjugate by the maximal suffix in $\widehat U \cdot \sphl(G_v,X_v)$,
shuffling this word past the maximal prefix in $\sphl(G_{\Ne(v)},X_{\Ne(v)})$,
and combining terms in $\widehat U$ and/or $\sphl(G_v,X_v)$.
At the end apply a final conjugation by a word over $X_{\Ne(v)}$
in order to apply Lemma~\ref{lem:UCb}(2) a last time.

After a finite number of iterations this process must stop,
resulting either in a word of the form $\dagger$,
or else in a word over one of the alphabets $X_{V \setminus \{v\}}$
or $X_{\Ne(v) \cup \{v\}}$.  In the latter case,
further conjugation shows that $g$ is conjugate to a word
of the form $\ddagger$.  

Finally, Proposition~\ref{prop:UC_geo} shows that all
words of the form $\dagger$ are conjugacy geodesics, 
and Proposition~\ref{prop:slsubset} shows that all words of 
the form $\ddagger$ are conjugacy geodesics,
for the group $G_V$ over the generating set $X_V$,
completing the proof of item (1).

For the proof of item (2), we start by noting that it is straightforward 
to check that if (i-ii) hold, then
$w_1=\tilde b u_1c_1 \cdots u_nc_n \sim_{G_V} w_2=\tilde b'u_1'c_1' \cdots u_{n'}'c_{n'}'$.

Now suppose that $w_1,w_2$ each have the form $\dagger$ or $\ddagger$
and represent conjugate elements of $G_V$.
Corollary~\ref{cor:ops_cgeo_x} shows that any two
conjugacy geodesics for $G_V$ over $X_V$ that
represent the same conjugacy class must have the same support.  
Hence either $w_1,w_2$ are both of the form $\ddagger$, 
in which case $w_1=w_2$ is the shortlex least representative of their 
conjugacy class in the subgroup, 
or both have the form~$\dagger$.

In the latter case, we write $w_1=\tilde b u_1c_1 \cdots u_nc_n$ and
$w_2=\tilde b'u_1'c_1' \cdots u_{n'}'c_{n'}'$ in $\dagger$ form, where the sequences $(\tilde b u_1),c_1,...,u_n,c_n$
and $(\tilde b'u_1'),c_1',...,u_{n'},c_{n'}$ are cyclically
reduced sequences of length at least 2.
The conjugacy theorem for free products with amalgamation
(see for example~\cite[Theorem~IV.2.8]{LS01}) implies that any two
cyclically reduced sequences of length at least 2
representing conjugate elements 
of the amalgamated product 
$G_V=G_{V \setminus \{v\}} *_{G_{\Ne(v)}} G_{\Ne(v) \cup \{v\}})$
must have the same length $n=n'$, 
and moreover there exist a $d \in \sphl(G_{\Ne(v)},X_{\Ne(v)})$
and an index $0 \le j \le n-1$ such that either
\begin{eqnarray}
w_2 &=_{G_V}& d(u_{j+1}c_{j+1} \cdots u_nc_n(\tilde b u_1)c_1 \cdots u_jc_j)d^{-1}
\label{eqn:conjugator} \text{ or} \\
w_2 &=_{G_V}& d(c_ju_{j+1}c_{j+1} \cdots u_nc_n(\tilde b u_1)c_1 \cdots u_j)d^{-1}.
\label{eqn:conjugatorwrong}
\end{eqnarray} 
We assume that $d$ has been chosen to be of minimal length; 
that is, no word of shorter length over $X_{\Ne(v)}$
satisfies Equation~\ref{eqn:conjugator} or~\ref{eqn:conjugatorwrong}.

If Equation~\ref{eqn:conjugatorwrong} holds, 
then since the support of $\tilde b$ is in the centralizing
sets of the supports of all of the $u_i$, we have 

$w_2=_{G_V} d\tilde b(c_ju_{j+1}c_{j+1} \cdots u_nc_nu_1c_1 \cdots u_j)d^{-1}$,

\noindent and then
Lemma~\ref{lem:UCb}(1) says that 

$w_2=_{G_V} (d\tilde b \hat d)c_j\hat u_{j+1}c_{j+1} \cdots \hat u_nc_n\hat u_1c_1 \cdots \hat u_j$
 
\noindent for a piecewise subword $\hat d$ of $d^{-1}$ and elements
$\hat u_1,...,\hat u_n \in \widehat U$.  Let $\hat b$ be the shortlex
least representative of $d\tilde b \hat d$.  Then
the normal form theorem for amalgamated products 
says that $\tilde b'=\hat b$ and $u_1' = c_j$ is the first
coset representative in the two representations of $w_2$.  However, 
this contradicts the fact that $u_1' \in \widehat U \setminus \{\emptyword\}$ 
and $c_j \in \sphl(G_v,X_v) \setminus \{\emptyword\}$, since these
sets are disjoint.  Hence
Equation~\ref{eqn:conjugator} must hold.

We now claim that $\supp(d) \subseteq \Ne(\cup_{i=1}^n \supp(u_i))$.  To
prove this claim, we suppose to the contrary that this containment
does not hold. 
Again using the fact that $\supp(\tilde b') \subseteq \Ne(\cup_{i=1}^n \supp(u_i))$
and Lemma~\ref{lem:UCb}(1), we have

$w_2=_{G_V} (d\tilde b \hat d)\hat u_{j+1}c_{j+1} \cdots \hat u_nc_n\hat u_1c_1 \cdots \hat u_jc_j$

\noindent for a proper piecewise subword $\hat d$ of $d^{-1}$ and elements
$\hat u_1,...,\hat u_n \in \widehat U$.
Note that $\wordl{\hat d} < \wordl{d}$.
Let $\hat b$ be the element of $\sphl(G_{\Ne(v)},X_{\Ne(v)})$ representing
$d\tilde b\hat d$.
Now the normal form theorem for amalgamated products 
says that 
$n=n'$,  
$u_i'=\hat u_{i+j}$ and $c_i'=c_{i+j}$ for all $i$ (where the indices
are considered modulo $n$), and $\tilde b'=\hat b$.  
Moreover, since $w_2$ is in the form $\dagger$,
we have $\supp(\hat b) \subseteq \Ne(\cup_{i=1}^n \supp(\hat u_i))$.
Hence
\begin{eqnarray*}
w_2&=_{G_V}& \hat u_{j+1}c_{j+1} \cdots \hat u_nc_n\hat u_1c_1 \cdots \hat u_jc_j \hat b \\
   &=_{G_V}& \hat u_{j+1}c_{j+1} \cdots \hat u_nc_n\hat u_1c_1 \cdots \hat u_jc_j(d\tilde b\hat d)\\
   &=_{G_V}& \hat d^{-1}(u_{j+1}c_{j+1} \cdots u_nc_n \tilde b u_1c_1 \cdots u_jc_j)\hat d,
\end{eqnarray*}
and so $\hat d^{-1}$ is a shorter word satisfying Equation~\ref{eqn:conjugator},
giving the required contradiction.

Now since
$\supp(d) \subseteq \Ne(\cup_{i=1}^n \supp(u_i))$,
then 

$w_2=(d\tilde bd^{-1})u_{j+1}c_{j+1} \cdots u_nc_nu_1c_1 \cdots u_jc_j$,

\noindent and the normal form theorem for amalgamated products 
says that $n=n'$, 
$u_i'=u_{i+j}$ and $c_i'=c_{i+j}$ for all $i$ (where the indices
are considered modulo $n$), and $\tilde b'=_{G_V} d\tilde b d^{-1}$.
Since both $\tilde b$ and $\tilde b'$ are in $\sphcl(G_{\Ne(v)},X_{\Ne(v)})$
and represent conjugate elements of $G_{\Ne(v)}$, then $\tilde b=\tilde b'$ 
as well. 
\end{proof}


\subsection{Equality of the standard and conjugacy growth rates}\label{subsec:equalrates}

In this section we show in Theorem~\ref{thmx:equal_rates} 
that the class of groups for which the standard and conjugacy growth
rates are equal is closed with respect to the graph product construction.
 
Recall that $\sphs_{(G,X)}(z)$ and
$\sphcs_{(G,X)}(z)$ denote the
\wsphs\ and \wsphcs, respectively,
for a group $G$ with respect to a generating set $X$.

\begin{notation}\label{notn:gpgrowthseries}
 Let $G_V$ be a graph product and assume that every vertex group $G_v$ has an inverse-closed generating set $X_v$. 
For each $V' \subseteq V$, let $X_{V'} := \cup_{v \in V'} X_v$ and 
 write
 \[
  \sigma_{V'}(z):=\sigma_{(G_{V'},X_{V'})}(z),\text{ and }\tilde{\sigma}_{V'}(z):=\tilde{\sigma}_{(G_{V'},X_{V'})}(z).
 \]
\end{notation}


We begin with a corollary of Corollary~\ref{cor:recursionforgrowth}.

\begin{corollary}\label{cor:standard_g_s_throw_a_vertex}
Let $G_V$ be a graph product group over a graph
with vertex set $V$, and let $v\in V$ be a vertex.
For each $v' \in V$ let $X_{v'}$ be an inverse-closed generating
set for the vertex group $G_{v'}$, and let $X_V = \cup_{v' \in V}X_{v'}$.  

Let $U=U_{G_{\Ne(v)} \backslash G_{V\setminus\{v\}}}$ be the 
admissible right transversal for $G_{\Ne(v)}$
in $G_{V\setminus\{v\}}$  with respect to the pair of
generating sets $(X_{V\setminus\{v\}},X_{\Ne(v)})$ 
from Lemma~\ref{lem:subgpadm},
and let $\sgs_U$ be the strict growth series
of the elements of 
$U$ with respect to $X_V$.
Using Notation~\ref{notn:gpgrowthseries}, then
\[
 \sgs_{V}=
\sgs_{\Ne(v)}\frac{\sgs_U\sgs_{\{v\}}}{\sgs_{\{v\}}+\sgs_U-\sgs_U\sgs_{\{v\}}}.
\]
Moreover, the radius of convergence of  $\sgs_{V}$ satisfies
\begin{align*}
\RC(\sgs_V)=\min\{&\RC(\sgs_{\Ne(v)}),\RC(\sgs_U),\RC(\sgs_{\{v\}}),\\
&\inf\{|z|\,:\, \sigma_{\{v\}}(z)+\sigma_U(z)-\sigma_U(z)\sigma_{\{v\}}(z)=0\}\}.
\end{align*}
\end{corollary}

\begin{proof}
If $V = \Ne(v) \cup \{v\}$ then $G_V=G_{v} \times G_{\Ne(v)}$ and
$U = \{\groupid\}$.
From Remark~\ref{rmk:admgrowth}, the spherical growth series
of a direct product of groups is the product of
the spherical growth series of the factors,
and so in this case we have $\sgs_U = 1$ and
$\sgs_V=\sgs_{\{v\}}\sgs_{\Ne(v)}$, as required.

Next assume that $V \neq \Ne(v) \cup \{v\}$ and so $U_{\Ne(v) \backslash (V\setminus\{v\})}\neq\{\groupid\}$.  
Note from Remark~\ref{rmk:admgrowth} that $\sgs_{V\setminus\{v\}}=\sgs_{\Ne(v)}\sgs_U$.
Corollary~\ref{cor:recursionforgrowth} and Lemma~\ref{lem:subgpadm}
give the required equality between the series.
Since the radius of convergence of a product is the minimum of the radii of convergence of the factors, we obtain the claim about $\RC(\sgs_V)$.
\end{proof}

\begin{remark}\label{rmk:rateisrcreciprocal}
Recall (Equation~(\ref{growth_rate}) in Section~\ref{subsec:cxseries}) 
that the 
exponential 
growth rate of the growth series of
a language $L$ over a finite set $X$ is the reciprocal
of the radius of convergence of the series; that is,
$$\egr_L = 1 / RC(F_L).$$
Thus for a group $G$ with generating set $X$ 
the spherical and spherical conjugacy growth rates
can be computed from the radii of convergence
of the corresponding growth series by
$\egrg = 1 / RC(\sphs)$ and $\egrc = 1 / RC(\sphcs)$.
\end{remark}

\begin{proposition}\label{prop:rate_comparison}
 Let $G_V$ be a graph product. For any set of vertices $V' \subseteq V$, 
the spherical conjugacy growth rates satisfy 
the inequality $\egrc(G_{V'},X_{V'}) \leq \egrc(G_{V},X_{V})$,
and the radii of convergence satisfy $\RC(\cgs_V)\leq \RC(\cgs_{V'})$.
\end{proposition}

\begin{proof}
Let $\slex$ be a  shortlex ordering on $X_V^*$ that is 
compatible with a total ordering $\ll$ on $V$ satisfying
$v' < v$ for all $v' \in V'$ and $v \in V \setminus V'$,
and let the shortlex ordering on
$X_{V'}^*$ be the restriction of the shortlex
ordering on $X_V^*$.
From Proposition~\ref{prop:slsubset}, we have
$\sphcl(G_V,X_V) \cap X_{V'}^* = \sphcl(G_{V'},X_{V'})$,
and in particular 
$\sphcl(G_{V'},X_{V'}) \subseteq \sphcl(G_V,X_V)$.  This
implies the inequality on 
exponential 
growth rates.  
Then Remark~\ref{rmk:rateisrcreciprocal} gives the inequality
for the radii of convergence.
\end{proof}

We are now ready to complete the proof of Theorem~\ref{thmx:equal_rates},
restated here with the notation from this section.


\medskip

\noindent{\bf Theorem~\ref{thmx:equal_rates}.}
\textit{
Let $G_V$ be a graph product group over a graph
with vertex set $V$ and assume that for each vertex $v\in V$ 
the spherical and spherical
conjugacy growth rates of $G_v$
are equal; that is,
$\egrg(G_v,X_v) = \egrc(G_v,X_v)$ for all $v \in V$.
Then 
$$\egrg(G_V,X_V) =\egrc(G_V,X_V)$$
and hence also $\RC(\sgs_V)=\RC(\cgs_V).$
}

\smallskip

\begin{proof}
Note that Remark~\ref{rmk:rateisrcreciprocal} shows that the equality
for the two growth rates follows from equality of the
two radii of convergence, and vice versa.
The proof is by induction on the number of vertices $|V|$. 
If $|V|=1$, the result is part of the hypothesis.
So assume $|V| \geq 2$.

Suppose that the graph $\Gamma$ underlying the graph
product is complete.  Then $G_V$ is the direct
product of the vertex groups and
the spherical and spherical conjugacy
growth series satisfy
$$
\sgs_{(G_V,X_V)}(z) = \prod_{v \in V} \sgs_{(G_v,X_v)}(z)
\hspace{.2in} \text{and} \hspace{.2in}
\cgs_{(G_V,X_V)}(z) = \prod_{v \in V} \cgs_{(G_v,X_v)}(z),
$$
so the radius of convergence of this product
is the minimum of the radii of convergence of the
factors; thus $\egrg(G_V,X_V) = \max\{\egrg(G_v,X_v) \mid v \in V\}$, and
$\egrc(G_V,X_V) = \max\{\egrc(G_v,X_v) \mid v \in V\}.$
Hence $\RC(\sgs_V)=\RC(\cgs_V)$ and
$\egrg(G_V,X_V)=\egrc(G_V,X_V)$ in this 
direct product case.

For the remainder of this proof we assume
that there are vertices $v,v'\in V$ such 
that $v$ and $v'$ are not connected by an edge.
By the induction hypothesis and Proposition~\ref{prop:rate_comparison}, we have
\begin{equation}\label{rates_inequality}
\RC(\cgs_V) \leq \RC(\cgs_{\Ne(v)}) = \RC(\sgs_{\Ne(v)}).
\end{equation}
Also by induction 
$\RC(\sgs_{\{v\}})=\RC(\cgs_{\{v\}})$, and so by Proposition~\ref{prop:rate_comparison} we have  
\begin{equation}\label{eq:rates_ineq_v}
\RC(\cgs_V) \leq \RC(\sgs_{\{v\}}).
\end{equation}

Let $\slex$ be a shortlex ordering on $X_V^*$
that is compatible with an ordering $\ll$ on $V$
satisfying $x\ll y$ for all $x \in \Ne(v)$ and 
$y \in V \setminus \Ne(v)$.
Let $\widehat U := \widehat U_{G_{\Ne(v)} \backslash G_{V \setminus \{v\}}}$ 
be the set representatives for an admissible transversal $U$
of $G_{\Ne(v)}$ in $G_{V \setminus \{v\}}$ with respect to
$(X_{V \setminus \{v\}},X_{\Ne(v)})$
defined in Lemma~\ref{lem:subgpadm}.
Since $\widehat U \subset \sphl(G_V,X_V)$,
the growth series satisfy $\sgs_U=F_{\widehat U}$.

Fix an element $d \in \sphl(G_v,X_v)$ of length 1, 
and consider the language 
$L=\{ud \mid u \in \widehat U \setminus \{\emptyword\}\}$.
Proposition~\ref{prop:UC_pearls} shows that 
distinct elements of $L$ represent distinct conjugacy classes. 
Hence the elements of $L$ of length $m$
are in bijection with the set of
conjugacy classes in $G_V$ represented by 
words in $L$ of length $m$; since  
Proposition~\ref{prop:UC_pearls} also shows that
the words in $L$ are conjugacy geodesics,
then the representatives in $\sphcl(G_V,X_V)$
of these conjugacy classes also have length $m$.
Hence the strict growth functions satisfy 
$\sgf_{\sphcl(G_V,X_V)}(m) \ge \sgf_{L}(m)
= \sgf_{\widehat U}(m-1)$
for all $m > 1$, and so the radii of convergence satisfy
\begin{equation}\label{eq:rates_ineq_U}
\RC(\cgs_V) \le \RC(\sgs_U).
\end{equation}

Similarly, consider the language 
$L=\{ u c \mid u \in \widehat U\setminus \{\emptyword\}, 
  c\in \sphl(G_v,X_v) \setminus \{\emptyword\}\}$.
Proposition~\ref{prop:UC_pearls} shows that the
elements of $\Necklaces(L)$ 
of length $m$ are in bijection with the 
conjugacy classes in $G_V$ represented by words of
the form $u_1c_1 \cdots u_nc_n$ of length $m$,
where each $u_i \in \widehat U \setminus \{\emptyword\}$
and $c_i \in \sphl(G_v,X_v) \setminus \{\emptyword\}$,
and Proposition~\ref{prop:UC_pearls}
shows that these words are also conjugacy geodesics.
Hence the strict growth functions satisfy 
$\sgf_{\sphcl(G_V,X_V)}(m) \ge \sgf_{\Necklaces(L)}(m)$
for all $m \ge 1$, and
therefore
$$
\RC(\cgs_V)\leq\RC(F_{\Necklaces(L)}).
$$ 
By Corollary \ref{cor_RC_necklace}, 
$\RC(F_{\Necklaces(L)})$
 is $\inf\{|z|\,:\,z\in\Co,\,|F_L(z)|=1\}$, 
and 
the growth series of $L$ in this case
is $F_L(z)=(\sgs_U(z)-1)(\sgs_{\{v\}}(z)-1)$.
Since $F_L(z)=1$ if and only if
$\sgs_{\{v\}}(z)+\sgs_U(z)-\sgs_U(z)\sgs_{\{v\}}(z)=0$, 
this yields 
\begin{equation}\label{eq:rates_inf}
\RC(\cgs_V)\leq \inf\{|z|\,:\,z\in\Co,\,
\sgs_{\{v\}}(z)+\sgs_U(z)-\sgs_U(z)\sgs_{\{v\}}(z)=0\}.
\end{equation}
In combination with  
inequalities~(\ref{rates_inequality}), (\ref{eq:rates_ineq_v}), (\ref{eq:rates_ineq_U}), and~(\ref{eq:rates_inf})
above, Corollary~\ref{cor:standard_g_s_throw_a_vertex}
shows that $\RC(\cgs_V)\leq \RC(\sgs_V)$.

On the other hand, since in any group the number of conjugacy classes 
represented by a conjugacy geodesic of a given length is 
at most the number of group elements of that length, $\RC(\cgs_V)\geq \RC(\sgs_V)$, 
yielding the equality of the two radii of convergence.
\end{proof}

The following result of Gekhtman and Yang~\cite[Corollary~1.3]{gy18} 
is also an immediate consequence of Theorem~\ref{thmx:equal_rates}.

\begin{corollary}\label{cor:raagracg}
Let $G$ be a right-angled Artin or Coxeter group; that is,
a graph product in which
the vertex groups are cyclic of infinite
order or of order 2, respectively. 
Then for the Artin or Coxeter generating set, respectively,
the spherical conjugacy 
growth rate of $G$ is the same as 
the spherical 
growth rate of $G$.
\end{corollary}


\subsection{The conjugacy growth series formula}\label{ss_conj_series}

In this section we prove Theorem~\ref{thmx:the_formula_for_cgs}, giving a recursive formula for the 
spherical conjugacy growth series $\tilde{\sigma}_V$ 
of a graph product group $G_V$ in terms of 
the spherical conjugacy and spherical growth series
$\tilde{\sigma}_{V'}$ and $\sigma_{V'}$ for the
subgraph products $G_{V'}$ where $V'\subsetneq V$.


We begin with an application of the 
inclusion-exclusion principle.
Given a graph product group $G_V$ on
a graph with vertex set $V$,
we view $\sphcs$ as a function $\sphcs:\Pa(V) \rightarrow \mathbb{Z}[[z]]$
to the ring of formal power series, 
where $\sphcs_S=\sphcs_{(G_S,X_S)}$ is the evaluation of 
$\sphcs$ at the subset $S \subset V$.
Recall from Proposition~\ref{prop:slsubset} that
for each $S \subseteq V$ the spherical conjugacy growth series
$\sphcs_S$ is the growth series of the language
$\sphcl(G_S,X_S) = \sphcl(G_V,X_V) \cap X_S^*$;
hence the series $\sphcs_S$ is also the contribution 
 in $\sphcs_V$ of the conjugacy classes 
having shortlex conjugacy representative
 with support contained in $S$. 

Define $f:\Pa(V) \rightarrow \mathbb{Z}[[z]]$  by setting
$f(T)$ to be the contribution in $\sphcs_V$ of the conjugacy classes 
having shortlex conjugacy representative with support exactly $T$.
Then for any subset $S \subseteq V$, we have
$\sphcs_S = \sum_{S' \subseteq S} f(S')$.
Now the M\"obius inversion principle 
(an extension of the principle of inclusion-exclusion; see for example~\cite[Example~3.8.3]{stanley}, \cite[Formula~3.1.2]{kry}) 
says that $f(S) = \sphcs^{\mathcal{M}}_{S}$,
where
$\sphcs^{\mathcal{M}}_{S} := \sum_{S'\subseteq S} (-1)^{|S|-|S'|} \sphcs_{S'}$ is the function that is the \emph{M\"obius inverse} of $f$, 
yielding the following.

\begin{lemma}\label{lem:exact_support}
Let $G_V$ be a graph product with generating set
$X_V$ and let $S\subseteq V$.
Let $\slex$ be a shortlex ordering on $X_V^*$ compatible
with a total ordering on $V$.
The contribution in $\sphcs_V$ of the conjugacy classes 
having shortlex conjugacy representative
 with support exactly $S$ is given by
 \[
\sphcs^{\mathcal{M}}_{S} = \sum_{S'\subseteq S} (-1)^{|S|-|S'|} \sphcs_{S'}.
 \]
\end{lemma}



Recall from Definition~\ref{def:neck} 
that   
 \[
  \neck(f)(z):=\sum_{k=1}^\infty\sum_{l=1}^\infty \frac{\phi(k)}{kl}\big(f(z^k)\big)^l=\sum_{k=1}^\infty \frac{-\phi(k)}{k}\Log(1-f(z^k))
 \]
for any complex power series $f$ with integer
coefficients satisfying $[z^0]f(z)=0$,
and recall from Proposition \ref{prop:neckseriesidentity} that the
function $\neck$ maps the growth series of a language $L$
to the growth series of the necklace language
$\Necklaces(L)$. 
 
The following paraphrased statement 
of Theorem~\ref{thmx:the_formula_for_cgs},
(using the notation above)
provides a recursive formula for computing
the conjugacy growth series of a graph product.


\medskip

\noindent{\bf Theorem~\ref{thmx:the_formula_for_cgs}.}
\textit{
Let $G_V$ be a graph product group over a graph
with vertex set $V$ and let $v\in V$ be a vertex. 
 Then the 
conjugacy growth series of $G_V$ is given by
\[  \cgs_V=\, \cgs_{V\setminus\{v\}}+\cgs_{\Ne(v)}(\cgs_{\{v\}}-1) +
   \sum_{S\subseteq \Ne(v)} \cgs^{\mathcal{M}}_S~ 
  \neck\left(\left(\frac{\sgs_{\Ne(S)\setminus\{v\}}}
    {\sgs_{\Ne(v)\cap\Ne(S)}}-1\right)(\sgs_{\{v\}}-1)\right).\]
Moreover, if $\{v\} \cup \Ne(v) = V$, then
$\cgs_V =  \cgs_{\Ne(v)}\cgs_{\{v\}}$. 
%
}

\smallskip

\begin{proof}
In the case that $\{v\} \cup \Ne(v) = V$, the graph product group
is a direct product $G_V = G_{\Ne(v)} \times G_{v}$, and
so the conjugacy growth series for $G_V$ is the product of
the corresponding series for the factors~\cite[Proposition~2.1]{CH14}.
Since $\Ne(v)=V \setminus \{v\}$, 
then the sets $\Ne(v) \cap \Ne(S)$ and $\Ne(S) \setminus \{v\}$
are equal, and so  
$\neck\left(\left(\frac{\sgs_{\Ne(S)\setminus\{v\}}}
    {\sgs_{\Ne(v)\cap\Ne(S)}}-1\right)(\sgs_{\{v\}}-1)\right)=
\neck(0)=0$.  Hence the theorem holds in this case.

For the remainder of this proof $\{v\} \cup \Ne(v) \neq V$. 
Let $\slex$ be a shortlex ordering on $X_V$ compatible with
an ordering $\ll$ on $V$ satisfying 
$x \ll y \ll v$ for all $x \in \Ne(v)$ and $y \in V \setminus (\{v\} \cup \Ne(v))$.
Let $\widehat U :=\widehat U_{G_{\Ne(v)} \backslash G_{V \setminus \{v\}}}$
be the set of representatives for the admissible transversal $U$ for $G_{\Ne(v)}$
in $G_{V \setminus \{v\}}$ with respect to
$(X_{V \setminus \{v\}},X_{\Ne(v)})$ from Lemma~\ref{lem:subgpadm}.
Propositions~\ref{prop:UC_pearls}
and~\ref{prop:slsubset}, together with the
fact that the shortlex conjugacy normal form set for the
direct product $G_{\Ne(v)} \times G_v$ is the concatenation of
the shortlex conjugacy normal form sets for the two factor groups,
show that $\cgs_V$ is equal to the 
growth series of the language
\[
\sphcl(G_{V \setminus \{v\}},X_{V \setminus \{v\}}) \bigsqcup  
\sphcl(G_{\Ne(v)},X_{\Ne(v)})[\sphcl(G_v,X_v) \setminus \{\emptyword\}] 
\bigsqcup L_\dagger
\]
over $X_V$,
where the language $L_\dagger$ is a set of
conjugacy class representatives 
containing exactly one word of the form $\dagger$, 
as defined in Proposition~\ref{prop:UC_pearls},
for each equivalence class with respect to the equivalence in 
Proposition~\ref{prop:UC_pearls}(2).
(Note that although we have not shown that the words in 
$L_\dagger$ are in $\sphcl(G_V,X_V)$, Proposition~\ref{prop:UC_pearls}
shows that they are conjugacy geodesic
representatives for their conjugacy classes.)
Hence $\cgs_V = \cgs_{V\setminus\{v\}}+\cgs_{\Ne(v)}(\cgs_{\{v\}}-1) + F_{L_\dagger}$,
where $F_{L_\dagger}$ is the growth series of the language $L_\dagger$.

Using Proposition~\ref{prop:UC_pearls}(2),
and the concept of necklaces from Section~\ref{subsec:necklaces},
the growth series of $L_\dagger$ 
equals the growth series of the disjoint union 
\[
 \bigsqcup_{S\subseteq \Ne(v)} 
  \{b\in \sphcl(G_{\Ne(v)},X_{\Ne(v)}) \,:\,\supp(b)=S\}\times
  \Necklaces(\widehat U_S~C),
\]
where 
$\widehat U_S := \widehat U \cap X_{\Ne(S)}^* \setminus \{\emptyword\}$
is the set of nonempty words in $\widehat U$ whose support is 
contained in $\Ne(S)$,
and $C = \sphl(G_v,X_v) \setminus \{\emptyword\}$.
The growth series of $ \{b\in \sphcl(G_{\Ne(v)},X_{\Ne(v)}) \,:\,\supp(b)=S\}$ 
is given by $\sphcs^{\mathcal{M}}_{S}$, from
Lemma~\ref{lem:exact_support} and Proposition~\ref{prop:slsubset}.  
The growth series
of the set $C = \sphl(G_v,X_v) \setminus \{\emptyword\}$ is $\sgs_v -1$.

By the definition of $\widehat U$ from Lemma~\ref{lem:subgpadm} we obtain
\begin{eqnarray*}
\widehat U_S &=& \left(\sphl(G_{V \setminus \{v\}},X_{V \setminus \{v\}})
  \cap X_{(V \setminus \{v\}) \setminus \Ne(v))}X_{V \setminus \{v\}}^* \right)
    \cap X_{\Ne(S)}^* \\
 &=& \sphl(G_{\Ne(S) \setminus \{v\}},X_{\Ne(S) \setminus \{v\}})
  \cap X_{(\Ne(S) \setminus \{v\}) \setminus (\Ne(v) \cap \Ne(S))}
  X_{\Ne(S) \setminus \{v\}}^*
\end{eqnarray*}
where the second equality follows from Proposition~\ref{prop:slsubset}.
Now Lemma~\ref{lem:subgpadm} shows that $\widehat U_S \cup \{\emptyword\}$ is
a set of shortlex representatives of the admissible
transversal for the subgroup
$G_{\Ne(v) \cap \Ne(S)}$
in $G_{\Ne(S) \setminus \{v\}}$  
with respect to 
$(X_{\Ne(S) \setminus \{v\}},
  X_{\Ne(v) \cap \Ne(S)})$.
Following the same counting 
argument as in Remark \ref{rmk:admgrowth},
admissibility of this transversal implies that the concatenation
$\sphl(G_{\Ne(v) \cap \Ne(S)},X_{\Ne(v) \cap \Ne(S)})(\widehat U_S \cup \{\emptyword\})$
is a set of (unique) geodesic representatives for
the elements of $G_{\Ne(S) \setminus \{v\}}$ over $X_{\Ne(S) \setminus \{v\}}$, and so
\[
F_{\widehat U_S} = \frac{\sgs_{\Ne(S)\setminus\{v\}}} 
  {\sgs_{\Ne(v) \cap \Ne(S)}} - 1
\]
(where as usual $F_{\widehat U_S}$ is the growth series of the language $\widehat U_S$).

Applying the growth series formula for necklaces in 
Proposition~\ref{prop:neckseriesidentity} and Definition~\ref{def:neck}, 
the contribution of $F_{L_\dagger}$ to $\cgs_V$ is
\[
 \sum_{S\subseteq \Ne(v)} \cgs^{\mathcal{M}}_S~ 
   \neck\left(\left(\frac{\sgs_{\Ne(S)\setminus\{v\}}}
    {\sgs_{\Ne(v) \cap \Ne(S)}}-1\right)(\sgs_{\{v\}}-1)\right).
\]
\end{proof}

We end this section with an example application of 
Theorems~\ref{thmx:the_formula_for_cgs} and~\ref{thmx:equal_rates}
to a right-angled Coxeter group.

\begin{example} \label{example_RACG}
Let $\Gamma = (V,E)$ be the finite simple graph
with vertex set $V = \{v_1,v_2,v_3,v_4\}$ and
edge set $E = \{\{v_1,v_2\}, \{v_2,v_3\}, \{v_3,v_4\}\}$;
that is, $\Gamma$ is a line segment made up of 3 edges.
For each $1 \le i \le 4$ let 
$G_{v_i} = \langle a_i \mid a_i^2 = 1\rangle$ be a
cyclic group of order 2 with inverse-closed generating set 
$X_{v_i} = \{a_i\}$.

We compute spherical and spherical conjugacy growth
series for several (virtually cyclic) subgraph products directly.
For each vertex group $G_{v_i}$
the growth series satisfy
    $\sgs_{\{v_i\}} = \cgs_{\{v_i\}} = 1+z$.  
The subgraph product $G_\emptyset$
is the trivial group with 
    $\sgs_\emptyset = 1$.
The group
$G_{V \setminus \{v_1\}}$ is the direct product of $G_{v_3}$ with
the infinite dihedral group $G_{\{v_2,v_4\}}$, 
and so the growth series satisfy
$\sgs_{V \setminus \{v_1\}} = \sgs_{\{v_2,v_4\}}\sgs_{\{v_3\}}$
and 
$\cgs_{V \setminus \{v_1\}} = \cgs_{\{v_2,v_4\}}\cgs_{\{v_3\}}$.
The series for the dihedral group are 
    $\sgs_{\{v_2,v_4\}}(z) = 1 + \frac{2z}{1-z} = \frac{1+z}{1-z}$  
and
    $\cgs_{\{v_2,v_4\}}(z) = \frac{1+2z-2z^3}{1-z^2}$.

We apply Corollary~\ref{cor:recursionforgrowth}
with the choice of vertex $v = v_1$.  Using the fact
that $\Ne(v) = \Ne(v_1) = \{v_2\}$,
we have
\begin{eqnarray*}
 \sgs_{V}&=&
\frac{\sgs_{\{v_2\}} \sgs_{V\setminus\{v_1\}}\sgs_{\{v_1\}}}
{\sgs_{\{v_2\}}\sgs_{\{v_1\}}
+\sgs_{V\setminus\{v_1\}}
-\sgs_{V\setminus\{v_1\}}\sgs_{\{v_1\}}} \\
&=& 
\frac{
(1+z) \left(\left(\frac{1+z}{1-z}\right)(1+z)\right)(1+z)
}
{
(1+z)^2
+\frac{(1+z)^2}{1-z}
-\left(\frac{(1+z)^2}{1-z}\right)(1+z)
}
 =
\frac{(1+z)^2}{1-2z}.
\end{eqnarray*}
Now Theorem~\ref{thmx:equal_rates} (or Corollary~\ref{cor:raagracg})
says that the radius of convergence of the spherical
conjugacy growth series  
is $\RC(\cgs_V)=\RC(\sgs_V)=\frac{1}{2}$, and so the
spherical conjugacy growth rate 
is $\egrc(G_V,X_V) = \egrg(G_V,X_V) = 2$.

To obtain an exact formula for $\cgs_V$ 
we apply Theorem~\ref{thmx:the_formula_for_cgs}
with $v = v_1$. 
Since $\Ne(v) = \{v_2\}$, $\Ne(\emptyset) = V$,
and $\Ne(v_2) = \{v_1,v_3\}$,  we have  
\begin{eqnarray*}  
\cgs_V &=& 
\cgs_{V\setminus\{v_1\}}
+\cgs_{\{v_2\}}(\cgs_{\{v_1\}}-1) 
+  \cgs^{\mathcal{M}}_\emptyset~ 
  \neck\left(\left(\frac{\sgs_{V \setminus\{v_1\}}}
    {\sgs_{\{v_2\}\cap V}}-1\right)(\sgs_{\{v_1\}}-1)\right) \\
 & & + \cgs^{\mathcal{M}}_{\{v_2\}}~ 
  \neck\left(\left(\frac{\sgs_{\{v_1,v_3\}\setminus\{v_1\}}}
    {\sgs_{\{v_2\}\cap\{v_1,v_3\}}}-1\right)(\sgs_{\{v_1\}}-1)\right).
\end{eqnarray*}
Computing the M\"obius inverses gives
    $\cgs^{\mathcal{M}}_\emptyset = (-1)^{0-0}\cgs_\emptyset = 1$
and 
    $\cgs^{\mathcal{M}}_{\{v_2\}} = (-1)^{1-1}\cgs_{\{v_2\}} + (-1)^{1-0}\cgs_\emptyset = 1+z-1 = z$. 
Plugging these and the series for the subgraph products
into the expression for $\cgs_V$ above yields
\[
\cgs_V = 
\left(\frac{1+2z-2z^3}{1-z^2}\right)(1+z)
+ (1+z)z
+ \neck\left(\left( \frac{2z}{1-z} \right) z\right)
+ z\neck(z^2).
\]
Now Example~\ref{ex:neckofpower} says that
$\neck(z^2) = \frac{z^2}{1-z^2}$, and so this simplifies to
\[
\cgs_V = 
\left(\frac{1+4z+3z^2-2z^3-3z^4}{1-z^2}\right)
+ \neck\left( \frac{2z^2}{1-z}\right).
\]
\end{example}


\section*{Acknowledgements}

The first and third author were supported by the Swiss National Science 
Foundation grant Professorship FN PP00P2-144681/1. The first author was also supported by EPSRC Standard Grant EP/R035814/1. The second
author was supported by grants from
the National Science Foundation (DMS-1313559) and
the Simons Foundation (Collaboration Grant number 581433). The third author was also supported by the FCT Project UID/MAT
/00297/2019
(Centro de Matem\'{a}tica e Aplica\c{c}\~{o}es) and the FCT Project PTDC/MHC-FIL/2583/2014.

\bigskip


\bibliographystyle{plain}   
\bibliography{cg1refs}


\bigskip

\small{

\textsc{Laura Ciobanu,
Mathematical and Computer Sciences,
Heriot-Watt University,      
Edinburgh EH14 4AS, UK}, \texttt{l.ciobanu@hw.ac.uk}

\bigskip

\textsc{Susan Hermiller,
Department of Mathematics,
University of Nebraska,
Lincoln, NE 68588-0130, USA}, \texttt{hermiller@unl.edu}

\bigskip

\textsc{V. Mercier,
Centro de Matem\'atica e Aplica\c{c}\~{o}es,
Faculdade de Ci\^{e}ncias e Tecnologia, Universidade Nova de Lisboa, 
2829-516 Caparica, Portugal}, \texttt{valen.mercier@gmail.com}}

\end{document}